\documentclass[11pt,reqno]{amsart}
\usepackage{amssymb,latexsym}
\usepackage[sorted,compressed-cites,sorted-cites,initials]{amsrefs}

\usepackage[cmtip,all]{xy}
\usepackage{geometry}
\usepackage{enumerate}
\usepackage{hyperref}
\numberwithin{equation}{section}
\usepackage{pifont,mathrsfs}
\makeatletter
\def\lort{{\mathrel{\hbox{\hglue .2ex
  \vrule \@height .07ex \@width 0.64ex
  \vrule \@height 1.28ex \@width .07ex
  \hglue .2ex}}}}

\def\rort{{\mathrel{\hbox{\hglue .4ex
  \vrule \@height 1.28ex \@width .07ex
  \vrule \@height .07ex \@width 0.64ex
  \hglue .2ex
}}}}
\def\lrort{{\mathrel{\hbox{\hglue .2ex
  \vrule \@height .07ex \@width 0.64ex
  \vrule \@height 1.28ex \@width .07ex\vrule \@height .07ex \@width 0.64ex
  \hglue .2ex}}}}

\makeatother

\newtheorem{theorem}{Theorem}[section]
\newtheorem{proposition}[theorem]{Proposition}
\newtheorem{conjecture}[theorem]{Conjecture}
\newtheorem{corollary}[theorem]{Corollary}

\newtheorem{lemma}[theorem]{Lemma}

\newtheorem{maintheorem}[theorem]{Main Theorem}
\theoremstyle{definition}
\newcommand{\Mat}[1]{M_{#1}(\kk)}
\def\QBin(#1,#2){\qbinom{#1}{#2}}
\DeclareMathOperator{\Ext}{Ext}
\DeclareMathOperator{\Hom}{Hom}
\DeclareMathOperator{\fin}{fin}

\DeclareMathOperator{\gr}{gr}

\DeclareMathOperator{\Iso}{Iso}
\DeclareMathOperator{\Ann}{Ann}
\DeclareMathOperator{\Aut}{Aut}
\DeclareMathOperator{\Ind}{Ind}
\DeclareMathOperator{\GL}{GL}

\DeclareMathOperator{\uExt}{\underline{Ext}}

\DeclareMathOperator{\Prim}{Prim}
\DeclareMathOperator{\QPrim}{QPrim}
\DeclareMathOperator{\Rep}{Rep}
\newcommand{\cat}[1]{{\mathscr #1}}

\newtheorem{remark}[theorem]{Remark}
\newtheorem{example}[theorem]{Example}
\newtheorem{definition}[theorem]{Definition}

\newcommand{\tensor}{\otimes}
\newcommand{\lr}[1]{\langle #1\rangle}

\def\FF{\mathbb{F}}
\def\ZZ{\mathbb{Z}}
\def\QQ{\mathbb{Q}}
\def\KK{\mathbb{K}}

\DeclareMathOperator{\Exp}{Exp}
\DeclareMathOperator{\End}{End}
\def\kk{\Bbbk}
\def\Im{\operatorname{Im}}

\newcommand{\qbinom}[3][q]{\genfrac[]{0pt}0{#2}{#3}_{#1}}

\advance\oddsidemargin by-0.5in
\advance\evensidemargin by-0.5in
\advance\textwidth by 1in

\begin{document}
\newgeometry{margin=2.5cm}
\author[A. Berenstein]{Arkady Berenstein}
\address{Department of Mathematics, University of Oregon,
Eugene, OR 97403, USA} \email{arkadiy@math.uoregon.edu}

\author[J. Greenstein]{Jacob Greenstein}
\address{Department of Mathematics, University of California, Riverside, CA 
92521.} 
 \email{jacob.greenstein@ucr.edu}

\date{\today}

\thanks{The authors were partially supported
by NSF grants DMS-0800247, DMS-1101507 (A.~B.) and DMS-0654421~(J.~G.)
and Simons foundation collaboration grant 245735~(J.~G.)}

\dedicatory{Dedicated to Professor Anthony Joseph on the occasion of his 
seventieth birthday}
\title{Primitively generated Hall algebras}

\begin{abstract} In the present paper we show that  Hall 
algebras of finitary exact categories behave like quantum 
groups in the sense that they are generated by indecomposable objects. Moreover,
for a large class of such categories, Hall algebras are generated by
their primitive elements, with respect to the natural comultiplication, even for 
non-hereditary categories. Finally, we introduce certain primitively
generated subalgebras of Hall algebras and conjecture an analogue of ``Lie 
correspondence''  for those finitary categories.
\end{abstract}

\maketitle

\tableofcontents

\section{Introduction}
\label{sect:intro}

It is well-known that quantum groups are not groups, but rather Hopf algebras 
which are similar to enveloping algebras of Lie algebras.
Hall-Ringel algebras $H_{\cat A}$ of finitary exact categories can be regarded, from many 
points of view, 
as generalizations of quantum groups. One aspect of this analogy is the following striking result which 
we failed to find in the literature
\begin{theorem}\label{th:PBW-property}
The Hall algebra~$H_{\cat A}$ of any finitary exact category~$\cat A$ is generated by
isomorphism classes of indecomposable objects in~$\cat A$.
\end{theorem}
We prove a refinement of this theorem (Theorem~\ref{thm:PBW-prop-Hall}), which is an analogue of the Poincar\'e-Birkhoff-Witt property for~$H_{\cat A}$,
in~\S\ref{pf:PBW}.

However, isomorphism classes of indecomposable objects are not the most efficient as a generating set. 
For example, if~$\cat A$ is the representation category 
of a (valued) Dynkin quiver~$Q$, then indecomposables correspond to all positive roots of the simple Lie algebra associated 
with~$Q$,
while~$H_{\cat A}$ can be generated by simple objects (in other words, indecomposables corresponding to simple roots of the Lie algebra). 
Having this in mind, we
introduce minimal generating sets for~$H_{\cat A}$, namely, primitive elements, which generalize these simple root generators.  

More precisely, for any finitary exact category~$\cat A$,
$H_{\cat A}$ has 
a natural coproduct $\Delta:H_{\cat A}\to H_{\cat A}\widehat\otimes H_{\cat A}$ whose 
image may lie in a suitable completion of the tensor square of~$H_{\cat A}$. 
Note however, that the multiplication and $\Delta$ are not always compatible, that 
is, $\Delta$ need not be a homomorphism of algebras. 
The compatibility is guaranteed  by {\em Green's theorem} (see~\cite{Green}) for all {\em hereditary} {\em cofinitary}
(so that~$\Delta$ is an ``honest'' comultiplication rather than a topological 
one) {\em abelian} categories $\cat A$. 
In particular, this includes all representation categories $\Rep_\kk Q$ where 
$Q$ is an acyclic (valued) quiver and~$\kk$ is a finite field. Moreover, for 
$\cat A=\Rep_\kk Q$, 
it was proved in a remarkable paper \cite{SV} that the Hall algebra $H_{\cat 
A}$ is generated by its space of primitive elements
$$
V_{\cat A}=\{v\in H_{\cat A}\,|\,\Delta(v)=v\otimes 1+1\otimes v\}.
$$
Thus, $H_{\cat A}$ is the Nichols algebra ${\mathcal B}(V_{\cat A})$ in an appropriate 
braided tensor category (see~\S\ref{subs:br-tens-cat} for details). 

We extend this result to a much larger class of categories that we refer to as 
{\it profinitary} categories. 
We introduce profinitary categories in terms of their {\em Grothendieck 
monoids} (denoted $\Gamma_{\cat A}$ for an exact category~$\cat A$, see~\S\ref{subs:Grothendieck-monoid} for 
precise definitions) by requiring that groups of morphisms between any two objects and
all Grothendieck equivalence classes  are finite. 
By 
definition, $H_{\cat A}$ is naturally graded by~$\Gamma_{\cat A}$ and
if $\cat A$ is profinitary, all homogeneous components $(H_{\cat A})_\gamma$, 
$\gamma\in \Gamma_{\cat A}$ are finite dimensional. 

The class of profinitary categories is large enough. For instance, it includes 
the abelian category $R-\fin $ of all {\it finite} $R$-modules $M$ (i.e., 
finite abelian groups with $R$-action) for a {\em finitary} unital ring $R$, as defined in~\cite{Ringel}*{Section~1}.
This includes all finitely generated (over~$\mathbb Z$) unital rings. 
Moreover, if ${\cat A}$ is profinitary, then so is any full sub-category ${\cat 
B}\subset {\cat A}$ closed under extensions.
The following is the main result of the present work.
\begin{maintheorem} 
\label{th:profinitary primitives}
For any  
profinitary and cofinitary exact category $\cat A$ the Hall algebra $H_{\cat 
A}$ is generated by the space $V_{\cat A}$ of its primitive elements.
Moreover, $V_{\cat A}$ is minimal in the sense that a non-zero element 
of~$V_{\cat A}$ cannot be expressed as a sum of products of elements 
of~$V_{\cat A}$. 
\end{maintheorem}

We prove Theorem~\ref{th:profinitary primitives} in~\S\ref{pf:main theorem}.

Based on the second assertion of Theorem~\ref{th:profinitary primitives}, we 
can introduce {\em quasi-Nichols algebras} as both algebras 
and coalgebras minimally generated by their primitive elements (see 
Definition~\ref{defn:quasi-Nichols} for details). In particular,
it is easy to see (cf. Lemma~\ref{lem:nichols-quasi-nichols}) that any Nichols 
algebra is quasi-Nichols. It is noteworthy
that the minimality of~$V_{\cat A}$ has the following nice consequence for 
constructing primitive elements in~$H_{\cat A}$:
once we found a subspace~$U$ of~$V_{\cat A}$ such that~$U$ generates $H_{\cat 
A}$ as an algebra, we must stop because $U$ is the 
space of {\em all} primitive elements in~$H_{\cat A}$.

\begin{remark} 
Similarly to Grothendieck groups, exact functors induce canonical homomorphisms 
of Grothendieck monoids. However, even for full embeddings,
such homomorphisms need not be injective.
On the other hand, unlike the Grothendieck group, the Grothendieck monoid 
always separates simple objects of the category.
For instance, if ${\cat A}$ is the category of $\kk$-representations of the 
quiver $Q=1\to 2$ with dimension vectors $(n,2n)$, $n\in \ZZ_{\ge 0}$, then 
$K_0({\cat A})\cong \ZZ$, but 
$\Gamma_{\cat A}$ is an additive monoid generated by $\beta_1,\beta_2$ subject 
to the relations $\beta_1+\beta_2=2\beta_1=2\beta_2$. 
The canonical homomorphism $\Gamma_{\cat A}\to K_0({\cat A})$ is given by 
$\beta_1\mapsto 1$, $\beta_2\mapsto 1$ and thus is not injective 
(see~\S\ref{subs:ex-1,2}
for details).
It should also be noted that in this example $\Gamma_{\cat A}$ is not a 
submonoid of the Grothendieck monoid of the category~$\Rep_\kk Q$
since in $\Gamma_{\Rep_\kk Q}$ both simple objects of~$\cat A$ belong to the 
same class.  
\end{remark}

A nice property of profinitary categories is that their Hall algebras always 
contain primitive elements.
If~$\cat A$ is profinitary, then its Grothendieck monoid admits a natural 
partial order  and
is generated by its minimal elements with respect to that order
(Proposition~\ref{prop:profinitary-ordered}). Moreover, for~$\gamma$ minimal 
the corresponding homogeneous component 
$(H_{\cat A})_\gamma$ of~$H_{\cat A}$ is one-dimensional and primitive.

Quite surprisingly, for a profinitary category, cofinitarity is a simple property of its Grothendieck monoid. 
We say that a monoid $\Gamma$ is locally finite if for all~$\gamma\in\Gamma$, the set 
$\{ (\alpha,\beta)\in\Gamma\times\Gamma\,:\,\alpha+\beta=\gamma\}$ is finite.
\begin{theorem}
\label{th:prof-cof}
A profinitary exact category~$\cat A$ is cofinitary if and only if~$\Gamma_{\cat A}$ 
is locally finite.
\end{theorem}
We prove this theorem in~\S\ref{subs:pf profinitary bifinitary}. As a corollary, we obtain two classes of 
categories for which
profinitarity implies cofinitarity.
\begin{corollary}
\begin{enumerate}[{\rm(a)}]
 \item\label{cor:prof-cof.a} Any full exact subcategory of a profinitary abelian category is cofinitary;
 \item\label{cor:prof-cof.c} Any profinitary exact category whose Grothendieck monoid is finitely generated is cofinitary.
\end{enumerate}\label{cor:prof-cof}
\end{corollary}
This Corollary is proven in~\S\ref{subs:pf profinitary bifinitary}. 
Based on the above, we propose the following conjecture.
\begin{conjecture}\label{conj:prof-cof}
For any profinitary exact category~$\cat A$ its Grothedieck monoid~$\Gamma_{\cat A}$ is locally finite.
\end{conjecture}
By Theorem~\ref{th:prof-cof}, any category as in the above Conjecture is also cofinitary.

This conjecture is non-trivial since there exist profinitary exact categories 
${\cat A}$ for which any ambient abelian category $\overline {\cat A}$ (which 
always exists, see e.g.~\cites{Buh,Kel}) is not profinitary, and the monoid $\Gamma_{\cat A}$
need not be finitely generated.

Theorem \ref{th:profinitary primitives} and Corollary~\ref{cor:prof-cof}\eqref{cor:prof-cof.a}
imply the following
\begin{theorem}\label{cor:hereditary nichols} If ${\cat A}$ is a profinitary 
hereditary abelian category, then $H_{\cat A}$ is a Nichols algebra (see Definition~\ref{defn:nichols}) of the 
(braided)  space $V_{\cat A}$ of its primitive elements.
\end{theorem}

We prove a refined version of this statement (Theorem~\ref{th:hereditary nichols}) in Section~\S\ref{pf:hereditary nichols}.

The case when $\cat A$ is the category of representations of an acyclic (valued) quiver over 
a finite field~$\kk$ is
established in~\cite{SV}*{Theorem~1.1}, which inspired the present work. If $\cat A$
is the category of nilpotent representations of~$\kk[x]$ for a finite field~$\kk$, then 
Theorem~\ref{cor:hereditary nichols} recovers the classical result of Zelevinsky (\cite{Zel})
that the Hall-Steinitz algebra is a Hopf algebra (see e.g.~\S\ref{subs:ex-jordan} for details).
More generally, since the category of finite dimensional $\kk$-representations of {\em any} finite quiver
is hereditary (see~\cites{Gab,Hub3}),
Theorem~\ref{cor:hereditary nichols} is applicable to such a category as well. In particular,
one can consider finite dimensional representations of a free algebra in~$n$ generators over~$\kk$.

Furthermore, by definition, $V_{\cat A}$ is 
graded 
$\Gamma_{\cat A}$, that is, 
$V_{\cat A}=\bigoplus_{\gamma\in \Gamma} (V_{\cat A})_\gamma$,  
so $\gamma\in\Gamma_{\cat A}$ with $(V_{\cat A})_\gamma\ne 0$ can be thought of as ``simple 
roots'' of ${\cat A}$. 
%An important consequence of Theorem~\ref{th:PBW-property} is the following
Given~$\gamma\in\Gamma_{\cat A}^+$, define its {\em multiplicity} $m_\gamma$ by
\begin{equation}\label{eq:root-mult}
m_\gamma:=\#\Ind\cat A_\gamma-\dim_\QQ(V_{\cat A})_\gamma.
\end{equation}
where $\Ind\cat A_\gamma=\Ind\cat A\cap \Iso\cat A_\gamma$. This definition is justified by the following
\begin{proposition}\label{prop:PBW}
Let~$\cat A$ be a profinitary cofinitary exact category. Then 
$m_\gamma\ge 0$
for all~$\gamma\in\Gamma_{\cat A}^+$.
\end{proposition}
We prove a more precise version of this result (Proposition~\ref{prop:PBW-prec}) in~\S\ref{subs:estimate}. In particular, 
Proposition~\ref{prop:PBW} implies that 
if $\Ind\cat A_\gamma=\emptyset$ then $(V_{\cat A})_\gamma=0$, that is, we should look for primitive 
elements only in those graded components where  indecomposables live. Moreover, if~$\Ind\cat A$ is finite,
then obviously $(V_{\cat A})$ is finite dimensional and we have an efficient procedure for computing it
(see Section~\ref{sec:examples}).

The term ``multiplicity'' is justified by the following 
result which is an immediate consequence of 
reformulations~\cite{Hua}*{Theorem~4.1} and~\cite{DX}*{\S4.1} of the famous Kac conjecture (\cite{Kac}), proved in~\cite{Haus}.
\begin{theorem}\label{thm:Kac-Hua-Haus}
Let $Q$ be an acyclic quiver,  
$\mathfrak g_Q$ be the corresponding Kac-Moody algebra
and $\cat A=\Rep_\kk(Q)$ where $\kk$ is a 
finite field with~$q$ elements.  
%Then 
%Retain the notation of Proposition~\ref{prop:SVh}. 
Then for any $\gamma\in\Gamma_{\cat A}$ 
one has
\begin{enumerate}[{\rm(a)}]
 \item\label{thm:KHH.a} $m_\gamma>0$ if and only if $\gamma$ is a non-simple positive root of~$\mathfrak g_Q$; in that case, 
 $m_\gamma=\dim(\mathfrak g_Q)_\gamma$, that is~$m_\gamma$ is the multiplicity of the root~$\gamma$ in~$\mathfrak g_Q$.
\item\label{thm:KHH.b} $(V_{\cat A})_\gamma=0$ unless $\gamma$ is simple or imaginary.
 \item\label{thm:KHH.c}  For any imaginary root~$\gamma$ of~$\mathfrak g_Q$, $\dim_\QQ(V_{\cat A})_\gamma=p_\gamma(q)$
 where $p_\gamma\in x\QQ[x]$.
 \end{enumerate}
\end{theorem}
% This theorem has the following consequences.
% \begin{corollary}
% Retain the assumptions of Theorem~\ref{thm:Kac-Hua-Haus}.
% \begin{enumerate}[{\rm(a)}]
% \item $(V_{\cat A})_\gamma=0$ unless $\gamma$ is a simple or imaginary root of~$\mathfrak g_Q$;
% \item 
% $\alpha\in \Gamma_{\cat A}$ is simple real if and only if $\alpha$ is a 
% simple root of ${\mathfrak  g}_Q$.
% \item $\alpha\in \Gamma_{\cat A}$ is simple imaginary if and only if $\alpha$ 
% is a positive imaginary root of ${\mathfrak  g}_Q$.
% \end{enumerate}\label{cor:Kac-Hua} 
% \end{corollary}

In view of Theorem~\ref{thm:Kac-Hua-Haus}\eqref{thm:KHH.c} and results of~\cite{SV} %that if $Q$ is an acyclic (valued) quiver, then all ``simple 
%roots'' of $\Rep_\kk Q$ are either simple roots or positive imaginary roots of 
%some Borcherds-Kac-Moody root system.   
we define {\em real simple roots} of ${\cat 
A}$ to be elements~$\gamma\in\Gamma_{\cat A}$ for which  $\dim_\QQ (V_{\cat A})_\gamma=1$ and 
{\it imaginary simple roots} of ${\cat A}$ to be those $\gamma\in \Gamma_{\cat A}$ with 
$\dim_\QQ (V_{\cat A})_\gamma\ge 2$. 
For a profinitary category~$\cat A$ we show (Lemma~\ref{lem:min-indec}) that
all minimal elements of~$\Gamma_{\cat A}\setminus\{0\}$ are real simple roots. 

In fact, the consideration of examples suggests that a 
stronger version of this statement holds.
\begin{conjecture}\label{conj:im-ineq} 
Let~$\cat A$ be a profinitary and cofinitary exact category.
Then each simple imaginary root of ${\cat A}$ has non-zero multiplicity.
\end{conjecture}   

Clearly, Theorem~\ref{thm:Kac-Hua-Haus} verifies this conjecture when~$\cat A=\Rep_\kk(Q)$ for any
finite acyclic quiver~$Q$.
We provide more supporting evidence in 
Section~\ref{sec:examples}. In those cases $m_\gamma=1$ quite frequently
(see~\S\ref{subs:ex-tube},
\S\ref{subs:ex-non-simply-laced} and
\S\ref{subs:equidim}). 
%Our examples also indicate that for $\kk$-linear profinitary abelian categories 
%with $|\kk|$ finite,
%$\dim_{\QQ}(V_{\cat A})_\gamma$ for~$\gamma$ imaginary simple is divisible 
%by~$|\kk|$ (cf.~Theorem~\ref{thm:Kac-Hua-Haus}\eqref{thm:KHH.c}).

Simple real roots are of special interest. Denote by $U_{\cat A}$ the 
subalgebra of $H_{\cat A}$ generated by 
all $V_\alpha$, where $\alpha$ runs over all real simple roots of ${\cat A}$ 
and refer to it as the {\it quantum enveloping algebra} of ${\cat A}$. 
The following well-known fact justifies this definition.

\begin{theorem}[\cite{Ringel1}] If $Q$ is an acyclic valued quiver, then $U_{\Rep_\kk Q}$ is 
isomorphic to a quantized enveloping algebra of the nilpotent part of 
${\mathfrak  g}_Q$.
\end{theorem}
Since~$[X]\in\Iso\cat A$ is primitive if and only if it is almost 
simple (see~Definition~\ref{def:almost-simple}), the algebra 
$U_{\cat A}$ contains the subalgebra~$C_{\cat A}$ of~$H_{\cat A}$ 
generated by isomorphism classes of all almost simple 
objects. We call $C_{\cat A}$ the {\em composition algebra}
of~$\cat A$ since it generalizes the composition algebra of~$\Rep_\kk Q$, which 
is the subalgebra of $H_{\Rep_\kk Q}$
generated by isomorphism classes of simple objects.
In fact, in the assumptions of the above Theorem, $U_{\Rep_\kk Q}=C_{\Rep_\kk 
Q}$. However, it frequently happens that $C_{\cat A}\subsetneq U_{\cat A}$
(see Section~\ref{sec:examples} for examples).
Note the following Corollary of Theorem~\ref{cor:hereditary nichols} and~\cite{AS}*{Corollary~2.3} (see Lemma~\ref{lem:incl-nichols}).
\begin{corollary}\label{cor:C and U Nichols}
If~$\cat A$ is a profinitary hereditary abelian category then both $C_{\cat A}$ and~$U_{\cat A}$ are
Nichols algebras.
\end{corollary}

It turns out that there is another algebra $E_{\cat A}$, which (yet 
conjecturally) ``squeezes'' between these two. That is, $E_{\cat A}$ is 
generated by elements $e_\gamma\in H_{\cat A}$, where $e_\gamma$ is the sum of 
all isomorphism classes of objects of ${\cat A}$ whose image in~$\Gamma$ is 
$\gamma$. Since 
$$\Exp_{\cat A}:=\sum_{\gamma\in \Gamma_A}e_\gamma$$ 
is a group-like element in the completion of~$H_{\cat A}$
with respect to a slightly different coproduct (see~\cite{BG}*{Lemma~A.1}), we 
referred to $\Exp_{\cat A}$ in~\cite{BG}
as {\it the exponential} of ${\cat A}$. Hence we sometimes refer to $E_{\cat 
A}$ as the exponential algebra of ${\cat A}$. 
By definition, $C_{\cat A}\subset E_{\cat A}$. 
\begin{conjecture} 
\label{conj:Lie}
For any profinitary category ${\cat A}$ one has
$$E_{\cat A}= U_{\cat A}\ ,$$
in particular, $\Exp_{\cat A}$ belongs to the completion of $U_{\cat A}$.
\end{conjecture} 

In Section~\ref{sec:examples} we provide several supporting examples of 
profinitary categories ${\cat A}$ together with the explicit presentations of 
$H_{\cat A}$, $U_{\cat A}$, and $E_{\cat A}$.

The significance of the conjecture is that it paves the ground for the ``Lie 
correspondence'' between the enveloping algebra 
$U_{\cat A}$ and the quantum Chevalley group $G_{\cat A}$ that we introduced 
in~\cite{BG} as an analogue of the corresponding Lie group. That is, the 
Conjecture \ref{conj:Lie} implies that the ``tame'' part of $G_{\cat A}$ 
belongs to the completion of~$U_{\cat A}$.

\subsection*{Acknowledgments}
We are grateful to Bernhard Keller, Dylan Ruppel and 
Vadim Vologodsky for stimulating discussions. An important part of this work 
was 
done during the first author's visit to the MSRI in the framework of
the ``Cluster algebras'' program and he thanks the Institute and the organizers 
for 
their hospitality and support. The second author is indebted to Department of 
Mathematics
of Uppsala University and especially Volodymyr Mazorchuk for their hospitality 
and support. On the final stage of work on this paper we benefited from
the hospitality of Institut des Hautes \'Etudes Scientiques (IHES), which we
gratefully acknowledge.

\section{Definitions and main results}

\subsection{Exact categories and Hall algebras}
All categories are assumed to be essentially small. For such a category~$\cat 
A$ we denote by $\Iso\cat A$ the set of isomorphism classes of 
objects in~$\cat A$.
We say that a category $\cat A$ is $\Hom$-finite if $\Hom_{\cat A}(X,Y)$ is a 
finite set for all $X,Y\in\cat A$.

Let~$\cat A$ be an exact category, in the sense of~\cites{Qu} (see 
also~\cites{Kel,Buh}). We denote by $\uExt^1_{\cat A}(A,B)$
the set of all isomorphism classes $[X]\in\Iso\cat A$ such that there exists a  
short exact sequence
\begin{equation}\label{eq:shex}
\xymatrix{B\ar@{>->}[r]^f &X\ar@{>>}[r]^g & A}
\end{equation}
(here
$f$ is a {\em monomorphism}, $g$ is an {\em epimorphism}, $f$ is a kernel 
of~$g$ and~$g$ is a cokernel of~$f$).
We say 
that $\cat A$ is {\em finitary} if it is $\Hom$-finite and $\uExt^1_{\cat 
A}(A,B)$ is finite for every $A,B\in\cat A$.

Following~\cite{Hub} we define Hall numbers for finitary exact categories as 
follows. For $A,B,X\in\cat A$ fixed, denote by
$\mathcal E(A,B)_X$ the set of all short exact sequences~\eqref{eq:shex}.
The group $\Aut_{\cat A}A\times\Aut_{\cat A}B$ acts freely on~$\mathcal 
E(A,B)_X$ by
$$
(\varphi,\psi).(f,g)=(f\varphi^{-1},\psi g),\qquad \varphi\in\Aut_{\cat 
A}B,\,\psi\in\Aut_{\cat A}A.
$$ 
The Hall number~$F_{AB}^X$ is the number of $\Aut_{\cat A}A\times\Aut_{\cat 
A}B$-orbits in~$\mathcal E(A,B)_X$ and equals to
$$
F_{AB}^X=\frac{\#\mathcal E(A,B)_X}{\#(\Aut_{\cat A}A\times\Aut_{\cat A}B)}.
$$
Denote
$$
H_{\cat A}=\mathbb Q\Iso\cat A=\bigoplus_{[X]\in\Iso\cat A}\mathbb Q\cdot [X].
$$
\begin{proposition}[Hall algebra, 
\cites{Ringel,Hub}]\label{P:multiplicationHall} 
For any finitary exact category~$\cat A$ the space~$H_{\cat A}$
is an associative unital $\mathbb Q$-algebra with the product given by:
\begin{equation}
\label{E:multiplicationHall}
[A]\cdot [B]=\sum_{[C]\in {\Iso\cat A}}  F_{A,B}^C [C].
\end{equation}
The unity $1\in  H_{\cat A}$ is the class $[0]$ of the zero object of ${\cat 
A}$.
\end{proposition}
It is well-known (see e.g.~\cites{Buh,Kel}) that each exact category~$\cat A$ 
can be realized as a full subcategory closed under extensions 
of an abelian category~$\overline{\cat A}$. However, even if~$\cat A$ was 
finitary, it might be impossible to find an ambient
abelian category which is also finitary. On the other hand, any full 
subcategory of a finitary abelian category closed under extensions is also
finitary. 

\subsection{Ordered monoids and PBW property of Hall algebras}\label{subs:ord-mon-PBW-prop}
Let~$\Lambda$ be an abelian monoid. 
We say that~$\Lambda$ is ordered if there exists 
a partial order~$\lhd$ on~$\Lambda$ such that 
\begin{enumerate}[$1^\circ.$]
 \item $0\lhd \lambda$ for all~$\lambda\in\Lambda^+:=\Lambda\setminus\{0\}$;
 \item $\mu\lhd \nu$, $\mu,\nu\in\Lambda^+\implies\lambda+\mu\lhd\lambda+\nu$ for all
$\lambda\in\Lambda$. 
\end{enumerate}

Let~$\cat A$ be a finitary exact category. 
The set~$\Iso\cat A$ is naturally an abelian monoid with the addition operation 
defined by~$[X]+[Y]=[X\oplus Y]$. The following property
is immediate.
\begin{lemma}[Cancellation property]\label{lem:can-prop}
Let $X,X',Y\in\cat A$. If $[X\oplus Y]=[X'\oplus Y]$ then~$[X]=[X']$.
\end{lemma}
Every object in~$\cat A$ is a 
finite direct sum of indecomposable objects (see Lemma~\ref{lem:fin-indec}). 
Thus, in particular, $\Ind\cat A$ generates~$\Iso\cat A$ as a monoid. The category~$\cat A$ is 
said to be Krull-Schmidt if~$\Iso\cat A$ is freely generated by~$\Ind\cat A$.

Define a relation~$\lhd$ on~$(\Iso\cat A)^+$ by $[M]\lhd[N]$ if~$[N]=[M^+\oplus M^-]$ and
there exists a non-split
short exact sequence $\xymatrix@C=2ex{M^-\ar@{>->}[r] &M\ar@{>>}[r] & M^+}$. By abuse of notation,
we denote by~$\lhd$ the transitive closure of this relation. We 
extend~$\lhd$ to~$\Iso\cat A$ by requiring that $[0]\lhd[M]$ for all~$[M]\in(\Iso\cat A)^+$.
\begin{proposition}\label{prop:Iso-ordered}
$(\Iso\cat A,\lhd)$ is an ordered monoid.
\end{proposition}
We prove this Proposition in~\S\ref{subs:part-ord-iso}. It is used as the key ingredient in 
a proof of the following theorem, which generalizes~\cite{GP}*{Theorem~3.1}.
\begin{theorem}\label{thm:PBW-prop-Hall}
Let~$\cat A$ be a finitary exact category. Then for any total order on the set~$\Ind\cat A$ of isomorphism 
classes of indecomposable objects in~$\cat A$, 
$H_{\cat A}$ is spanned, as a $\QQ$-vector space, by ordered monomials on~$\Ind\cat A$. Moreover,
if~$\cat A$ is Krull-Schmidt, then such monomials form a basis of~$H_{\cat A}$.
\end{theorem}
We prove this Theorem in~\S\ref{pf:PBW}.
After~\cites{Joyce,Rie}, this further extends an analogy between Hall algebras of finitary categories and universal
enveloping algebras.

\subsection{Grothendieck monoid and grading}\label{subs:Grothendieck-monoid}
Define the relation~$\,\equiv\,$ on the monoid~$\Iso\cat A$ by 
$$
[X]\equiv [Y]\iff \text{$[X],[Y]\in\uExt^1_{\cat A}(M,N)$ for some $M,N\in\cat 
A$}.
$$
This relation is clearly symmetric and reflexive, hence its transitive closure 
is an equivalence relation
on~$\Iso\cat A$
which we also denote by $\equiv$. The additivity of~$\Ext^1_{\cat A}(A,B):=
\bigcup_{X} \mathcal E(A,B)_X/\Aut_{\cat A}X$
in both~$A$ and~$B$ yields the following
\begin{lemma}
The relation $\,\equiv\,$ is a congruence relation on~$\Iso\cat A$, that is
$[X]\equiv[Y]$, $[X']\equiv[Y']$ implies that $[X\oplus X']\equiv [Y\oplus 
Y']$. 
\end{lemma}
\begin{definition}
The Grothendieck monoid $\Gamma_{\cat A}$ of~$\cat A$ is the quotient of 
$\Iso\cat A$ by the congruence~$\equiv$.
\end{definition}
Given an object~$M$ in~$\cat A$, we denote its image 
in~$\Gamma_{\cat A}$ by~$|M|$.  For all~$\gamma\in\Gamma_{\cat A}$, set
$$
\Iso\cat A_\gamma=\{ [X]\in\Iso\cat A\,:\, |X|=\gamma\}.
$$
We will refer to $\Iso\cat A_\gamma$ as a Grothendieck class in~$\cat A$. We denote~$\Ind\cat A_\gamma=
\Ind\cat A\cap \Iso\cat A_\gamma$.

The following fact is obvious.
\begin{lemma}
For any finitary exact category~$\cat A$,
the assignment $[M]\mapsto |M|$ defines a grading of the Hall algebra~$H_{\cat 
A}$ of~$\cat A$ by the 
Grothendieck monoid~$\Gamma_{\cat A}$.
\end{lemma}

\begin{remark}
After Grothendieck, one defines the Grothedieck group~$K_0(\cat A)$ of~$\cat A$ as the 
universal abelian group generated by~$\Gamma_{\cat A}$.
Note that the canonical homomorphism of monoids $\Gamma_{\cat A}\to K_0(\cat 
A)$ can be very far from being injective.
One
example was already provided in the Introduction. Perhaps, the most extreme example is 
the following.
Let $\cat A=\operatorname{Vect}_\kk$ be the category of {\em all} $\kk$-vector spaces 
over some field~$\kk$. Then~$\Gamma_{\cat A}$ identifies with the monoid 
of cardinal numbers. In particular, if~$V$ is infinite dimensional and 
$W$ is finite dimensional then $|V|=|V|+|V|=|W|+|V|$.
This implies that in~$K_0(\operatorname{Vect}_\kk)$, $|U|=0$ for every object~$U$ of $\operatorname{Vect}_\kk$, that is $K_0(\operatorname{Vect}_\kk)=0$.
\end{remark}

Also, while $K_0(\cat A)$ can contain elements of finite 
order, this never occurs in~$\Gamma_{\cat A}$.
Indeed, since $[0]\in\uExt^1_{\cat A}(A,B)$ implies that $A=B=0$ and the direct 
sum of two non-zero objects is clearly non-zero, we immediately obtain the 
following
\begin{lemma}\label{lem:unique inv}
For any exact category~$\cat A$, zero is the only invertible element of the 
Grothendieck monoid~$\Gamma_{\cat A}$. 
\end{lemma}

\subsection{Profinitary and cofinitary categories}\label{subs:order-gr-mon}
Let~$\Gamma$ be an abelian monoid. Define a relation $\preceq$ on~$\Gamma$ by 
$\alpha\preceq \beta$ if $\beta=\alpha+\gamma$ for some~$\gamma\in\Gamma$.
This relation is clearly an additive preorder and $0\preceq \gamma$ for any 
$\gamma\in\Gamma$. 
The following Lemma is obvious.
\begin{lemma}\label{lem:nat-ord}
The preorder $\preceq$ is a partial order on~$\Gamma$ if and only if 
the equality $\alpha+\beta+\gamma=\alpha$, $\alpha,\beta,\gamma\in\Gamma$, 
implies that
$\alpha=\alpha+\beta=\alpha+\gamma$. In that case, $0$ is the only invertible 
element of~$\Gamma$.
\end{lemma}
We say that $\Gamma$ is {\em naturally ordered} if $\,\preceq\,$ is a partial 
order.

\begin{definition}\label{def:prof-cof-bif}
We say that a $\Hom$-finite exact category~$\cat A$ is 
\begin{enumerate}[{(i)}]
\item {\em profinitary}  if $\Iso\cat A_\gamma$ is a finite set for 
all~$\gamma\in\Gamma_{\cat A}$.
\item {\em cofinitary} (cf.~\cite{KSV}) if for every $[X]\in\Iso\cat A$, the 
set 
$$
\{([A],[B])\in\Iso\cat A\times\Iso\cat A\,:\, [X]\in\uExt^1_{\cat A}([A],[B])\}
$$
is finite.
\end{enumerate}
\end{definition}
Since~$\mathcal E(M,N)_X$ identifies with a subset of~$\Hom_{\cat A}(N,X)\times\Hom_{\cat A}(X,M)$,
any profinitary category is necessarily finitary.
\begin{proposition}\label{prop:profinitary-ordered}
Let~$\cat A$ be a profinitary category.
Then $\Gamma_{\cat A}$ is naturally ordered and is generated by its minimal 
elements.
\end{proposition}
A proof of this proposition is given in~\S\ref{pf:prop-profinitary-ordered}.
\begin{remark}
One can characterize profinitary categories as follows. If~$\cat A$ is $\Hom$-finite and its Grothendieck monoid 
is locally finite, as defined before Theorem~\ref{th:prof-cof}, and $\Ind\cat A_\gamma$ is finite for all~$\gamma\in\Gamma_{\cat A}$
then $\cat A$ is profinitary.
\end{remark}

\begin{theorem}\label{th:prof-ab}
Any profinitary abelian category has the finite length property hence is Krull-Schmidt. 
\end{theorem}
We prove this theorem in~\S\ref{subs:pf profinitary bifinitary}. This result, 
together with Theorem~\ref{th:prof-cof}, yields Corollary~\ref{cor:prof-cof}\eqref{cor:prof-cof.a}.
\begin{remark}
The finite length property in an abelian category~$\cat A$ is much stronger 
than the Krull-Schmidt property. For instance,
the Grothendieck monoid of an abelian category with the finite length property 
is freely generated by classes of simple 
objects and the canonical homomorphism $\Gamma_{\cat A}\to K_0(\cat A)$ is 
injective.
On the other hand, the category of coherent sheaves on~$\mathbb P^1$ is 
Krull-Schmidt, but lacks the
finite length property and each Grothendieck class $\Iso\cat A_\gamma$, 
$\gamma\not=0$ is infinite.
\end{remark}

\subsection{Comultiplication and primitive generation}
Define a linear map $\Delta:H_{\cat A} \to H_{\cat A}  \widehat\otimes H_{\cat 
A}$
\begin{equation}
\label{eq:coproductDelta}
\Delta([C])= \sum_{[A],[B]\in {\Iso\cat A}}  F^{A,B}_C\cdot  [A]\otimes [B]\ ,
\end{equation}
where $H_{\cat A}\widehat\tensor H_{\cat A}$ is the completion of the usual 
tensor product with respect to the~$\Gamma_{\cat A}$-grading and
$F^{A,B}_C$ is the dual Hall number given by 
$$
F^{A,B}_C=\frac{\#(\Aut_{\cat A}A\times\Aut_{\cat A}B)}{\#\Aut_{\cat A}C} \, 
F_{B,A}^C.
$$
It follows from Riedtmann's formula~\cite{Rie} that
$$
F^{A,B}_C=\frac{\#\Ext^1_{\cat A}(B,A)_C}{\#\Hom_{\cat A}(B,A)},
$$
where $\Ext^1_{\cat A}(B,A)_C=\mathcal E(B,A)_C/\Aut_{\cat A}C$.
Also define a linear map $\varepsilon:H_{\cat A}\to \QQ$  by
\begin{equation}
\label{eq:counit}
\varepsilon([C])= \delta_{[0],[C]}.
\end{equation}
The following fact is obvious. 
\begin{lemma}\label{lem:coprod}
\begin{enumerate}[{\rm(a)}]
 \item $H_{\cat A}$ is a topological coalgebra with respect to the above 
comultiplication and counit;
\item If~$\cat A$ is cofinitary then $H_{\cat A}$ is an ordinary coalgebra,
that is, the image of the comultiplication $\Delta$ lies in~$H_{\cat A}\tensor 
H_{\cat A}$.
\end{enumerate}
\end{lemma}

For any coalgebra $C$  with unity denote by $\Prim(C)$ the set of all primitive 
elements, i.e., 
$$\Prim(C)=\{c\in C:\Delta(c)=c\otimes 1+1\otimes c\}\ .$$ 
\begin{definition}\label{defn:quasi-Nichols}
Let $A$ be both a unital algebra and a coalgebra over a field $\mathbb F$. We 
say that $A$ is a {\em quasi-Nichols algebra}
if $A$ decomposes as $\mathbb F\oplus V\oplus \big(\sum_{r>1} V^r\big)$ where 
$V=\Prim(A)$.
\end{definition}

The following is the main result of the paper (Theorem~\ref{th:profinitary 
primitives}) and is proven 
in~\S\ref{pf:main theorem}.
\begin{theorem} 
\label{th:prim generation}
Let~$\cat A$ be a profinitary and cofinitary exact category. Then 
the Hall algebra $H_{\cat A}$ is quasi-Nichols.
\end{theorem}
This Theorem has the following useful corollary, which we prove in~\S\ref{subs:estimate}.
\begin{corollary}\label{cor:prim-gen}Let
\begin{equation}\label{eq:P-defn}
P=\ker\varepsilon\cdot\ker\varepsilon=\QQ\{ [M][N]\,:\, [M],[N]\in(\Iso\cat A)^+\},\qquad P_\gamma:=P\cap (H_{\cat A})_\gamma.
\end{equation}
Then $P=\sum_{k\ge 2} \Prim(H_{\cat A})^k=\sum_{k\ge 2} (\QQ\Ind\cat A)^k$ and
$(H_{\cat A})_\gamma=\Prim(H_{\cat A})_\gamma\oplus P_\gamma$ for all~$\gamma\in\Gamma_{\cat A}^+$.
\end{corollary}

A natural question is to compute dimensions of~$\Prim(H_{\cat A})_\gamma$, $\gamma\in\Gamma^+_{\cat A}$.
The following is a refinement of Proposition~\ref{prop:PBW}.
\begin{proposition}\label{prop:PBW-prec}
In the notation~\eqref{eq:root-mult} we have  
$$
m_\gamma=\dim_\QQ (P_\gamma\cap \QQ\Ind\cat A_\gamma)
$$
for all~$\gamma\in\Gamma_{\cat A}^+$.
%$$
%\dim_\QQ \Prim(H_{\cat A})_\gamma=\#\Ind\cat A_\gamma-\dim_\QQ (P_\gamma\cap \QQ\Ind\cat A_\gamma). 
%$$
In particular, if~$\Ind\cat A_\gamma\subset P_\gamma$ then
$\Prim(H_{\cat A})_\gamma=0$.
\end{proposition}
We prove Proposition~\ref{prop:PBW-prec} in~\S\ref{subs:estimate}, as well as the following observation, which is useful 
for computing primitive elements.
\begin{lemma}\label{lem:prim}
Each primitive element contains at least one~$[X]\in\Ind\cat A$ in its decomposition with respect to the basis~$\Iso\cat A$
of~$H_{\cat A}$. In other words, $\Prim(H_{\cat A})\cap \QQ(\Iso\cat A\setminus\Ind\cat A)=\{0\}$.
\end{lemma}

\subsection{Hereditary categories and Nichols algebras}\label{subs:br-tens-cat}
Let $\Gamma$ be an abelian monoid and let ${\cat C}_\Gamma$ be the tensor 
category of $\Gamma$-graded vector spaces $V=\bigoplus_{\gamma\in \Gamma} 
V_\gamma$.
The following fact can be easily checked.

\begin{lemma}\label{lem:bichar-category} For each bicharacter  
$\chi:\Gamma\times\Gamma\to \FF^\times$  the category ${\cat C}_\Gamma$ is a 
braided tensor category~$(\cat C_\Gamma,\Psi)$ with the braiding  $\Psi_{U,V}:U\otimes V\to V\otimes 
U$ for objects $U,V$ in ${\cat C}_\Gamma$ given by
$$\Psi_{U,V}(u\otimes v)=\chi(\gamma,\delta) \, v\otimes u,$$
for any $u\in U_\gamma$, $v\in V_\delta$, $\gamma,\delta\in\Gamma$.   
\end{lemma}

By a slight abuse of notation, given a bicharacter $\chi:\Gamma\times\Gamma\to 
\FF^\times$  we denote this braided tensor category ${\cat C}_\Gamma$  by  
${\cat C}_\chi$.

Now let ${\cat A}$ be a finitary {\em hereditary}  category, i.e., 
$\Ext^i_{\cat A}(M,N)=0$ for $i>1$ and all $M,N\in {\cat A}$. 
Let $\chi_{\cat A}: \Gamma\times \Gamma\to \QQ^\times$ be the bi-character given by
$$\chi_{\cat A}(|M|,|N|)=\frac{\#\Ext^1_{\cat A}(M,N)} {\#\Hom_{\cat A}(M,N)}.$$
The bicharacter $\chi_{\cat A}$ is easily seen to be well-defined because 
it is just the (multiplicative) Euler form.

Nichols algebras were formally defined in~\cite{AS}. 
\begin{definition}[\cite{AS}*{Definition~2.1}]\label{defn:nichols}
Let~$(\cat C,\Psi)$ be a braided $\FF$-linear tensor category with a braiding~$\Psi$.
Let~$V$ be an object in~$(\cat C,\Psi)$. A graded bialgebra with unity $B=\bigoplus_{n\ge 0} B_n$
in~$(\cat C,\Psi)$ is called a 
{\em Nichols algebra} of~$V$ if $B_0=\FF$, $B_1=V$ and $B$ is generated, as an algebra, by $B_1=\Prim(B)$.
\end{definition}

For each object $V$ of a braided tensor category $(\cat C,\Psi)$ the tensor algebra 
$T(V)$ is a graded bialgebra (even a Hopf algebra) in $(\cat C,\Psi)$ 
with the coproduct determined by requiring each $v\in V$ to be primitive and the grading 
defined by assigning degree~$1$ to elements of~$V$.
It is well-known (\cite{AS}*{Proposition~2.2}) that the Nichols algebra of~$V$ is 
unique up to an isomorphism and is   
 the quotient of $T(V)$ by the maximal graded bi-ideal $\mathfrak I$ of~$T(V)$ which is an object in $(\cat C,\Psi)$ 
and satisfies $\mathfrak I\cap V=\{0\}$. Henceforth we denote the Nichols algebra of~$V$ by~$\mathcal B(V)$. 

The following is proved in~\cite{AS}*{Corollary~2.3}
\begin{lemma}\label{lem:incl-nichols}
The assignment $V\mapsto \mathcal B(V)$ defines a functor from~$(\cat C,\Psi)$ to the category of bialgebras in~$(\cat C,\Psi)$.
Moreover, for any morphism $f:U\to V$ in~$(\cat C,\Psi)$, the kernel of the corresponding homomorphism $\mathcal B(f)$
is the (bi)ideal in~$\mathcal B(U)$ generated by~$\ker f\subset U$.
\end{lemma}

The following fact is immediate from the definitions.
 \begin{lemma}\label{lem:nichols-quasi-nichols}
Let $B$ be a bialgebra in~$(\cat C,\Psi)$ which is a quasi-Nichols algebra. Then $B$ is Nichols if 
and only if $\sum_{r\ge2} (\Prim(B))^r$ is direct.
\end{lemma}

The following extends the main result of~\cite{SV}.
\begin{theorem} 
\label{th:hereditary nichols}
For any profinitary hereditary abelian category ${\cat A}$ the Hall algebra 
$H_{\cat A}$ is isomorphic to the Nichols algebra $\mathcal B(V_{\cat A})$ 
in the category~$\cat C_{\chi_{\cat A}}$ where $V_{\cat A}=\Prim(H_{\cat A})$. 
\end{theorem}
We prove this Theorem in~\S\ref{pf:hereditary nichols}.

\begin{remark}
In fact, the original result of~\cite{SV} (Theorem~1.1) follows from 
Theorem~\ref{th:hereditary nichols}. The classification of
diagonally braided Nichols algebras was obtained in~\cite{AS}*{Section~5} and, 
in particular, generalizes
some results of~\cite{SV}.
\end{remark}

\section{Examples}\label{sec:examples}

In this section we construct primitive elements in several Hall algebras and 
provide supporting evidence for 
Conjectures~\ref{conj:Lie} and~\ref{conj:im-ineq}.  
Throughout this section we write $\bar\Delta(x)=\Delta(x)-x\tensor 1-1\tensor 
x$ (thus, $x$ is primitive if and only if~$x\in\ker\bar\Delta$).
Needless to say, every (almost) simple object~$S$ satisfies $\bar\Delta([S])=0$ so we 
focus only on non-simple primitive elements.
In this section, $\kk$ always denotes a finite field with~$q$ elements and all 
categories will be assumed to be $\kk$-linear.

\subsection{Classical Hall-Steinitz algebra}\label{subs:ex-jordan}
Let $R$ be a principal ideal domain such that $R/\mathfrak m$ is a finite field 
for any maximal ideal~$\mathfrak m$ of~$R$. Let $\cat A=\cat A(\mathfrak m)$
be the full subcategory of 
finite length $R$-modules~$M$ satisfying $\mathfrak m^r M=0$ for some~$r\ge 0$. 
Then for each~$r>0$,
there exists a unique, up to an isomorphism, indecomposable object $\mathcal 
I_r=R/\mathfrak m^r\in\cat A$.
More generally, given a partition~$\lambda=(\lambda_1\ge \cdots\ge\lambda_k>0)$, 
set $\mathcal I_{\lambda}=\mathcal I_{\lambda_1}\oplus \cdots\oplus \mathcal I_{\lambda_k}$
and write~$\ell(\lambda)=k$.

Since the Euler form of~$\cat A$ is identically zero and~$\cat A$ is 
hereditary, $H_{\cat A}$
is an ordinary Hopf algebra (the braiding is trivial). The Grothendieck monoid of~$\cat A$ being~$\mathbb Z_{\ge 0}$, the algebra 
$H_{\cat A}$ is $\mathbb Z_{\ge 0}$-graded.
We now provide a new 
(very short) proof of the following classical result. 
\begin{theorem}[\cites{Mac,Zel}]
The Hall algebra~$H_{\cat A}$ is commutative and co-commutative and is freely 
generated by the~$[\mathcal I_n]$, $n>0$. Moreover, $H_{\cat A}$ is freely generated by its primitive 
elements~$\mathcal P_n$, $n>0$.
\end{theorem}
\begin{proof}
Let~$\cat A=\cat A$.
It is easy to see, using duality, that~$H_{\cat A}$ is commutative, hence co-commutative. Let~$\mathscr P$ be the set of 
all partitions.
Given a 
partition~$\lambda=(\lambda_1\ge\cdots\ge\lambda_r>0)\in\mathscr P$, let $M_\lambda=[\mathcal I_{\lambda_1}]\cdots[\mathcal I_{\lambda_r}]$.
By
Theorem~\ref{thm:PBW-prop-Hall}, the set $\{ M_\lambda\}_{\lambda\in\mathscr P}$ is a basis of~$H_{\cat A}$, 
hence $H_{\cat A}$ is freely generated by the isomorphism classes of indecomposables $[\mathcal I_n]$, $n>0$.
Since~$H_{\cat A}$ is commutative, 
$P=\ker\varepsilon\cdot\ker\varepsilon$ is spanned by the $M_\lambda$ with~$\ell(\lambda)\ge 2$, 
hence $\QQ\Ind\cat A\cap P=\{0\}$ and by Proposition~\ref{prop:PBW-prec}
$\dim_\QQ\Prim(H_{\cat A})_n=\#\Ind\cat A_n=1$ for all~$n>0$. Thus, for each~$n>0$ we have a unique, up
to a scalar, non-zero primitive element~$\mathcal P_n$ in~$(H_{\cat A})_n$.
The dimension considerations and Theorem~\ref{th:prim generation}
immediately imply that~$H_{\cat A}$ is freely generated by the $\mathcal P_n$, $n>0$.
\end{proof}
This theorem has the following nice Corollary.
\begin{corollary}\label{cor:class-Hall}
For all~$n>0$, let $x_n\in (H_{\cat A})_n\setminus \QQ(\Iso\cat A_n\setminus \Ind\cat A_n)$.
Then $\{x_n\}_{n>0}$ freely generates $H_{\cat A}$. In particular, $E_{\cat A}=H_{\cat A}$.
\end{corollary}

The elements~$\mathcal P_n$ can be computed explicitly (see e.g.~\cite{Hub1}*{Section~5}), namely 
$$
\mathcal P_n=\sum_{\lambda\vdash n} 
\Big(\prod_{j=1}^{\ell(\lambda)-1}(1-q^j)\Big) [\mathcal I_\lambda],
$$
where $q=|R/\mathfrak m|$. 

Under the isomorphism $\psi:H_{\cat A}\to \operatorname{Sym}$, $[\mathcal I_\lambda]\mapsto 
q^{-n(\lambda)} P_\lambda(x;q^{-1})$ (\cites{Mac,Zel}), where 
$\operatorname{Sym}$ is the algebra of symmetric
polynomials in infinitely many variables
and $P_\lambda(x;t)$ is the Hall-Littlewood polynomial, the image of~$\mathcal P_n$
is the $n$th power sum~$p_n$.
As shown in~\cite{Zel}, the~$p_n$ are 
primitive elements in~$\operatorname{Sym}$
with the comultiplication defined by 
$$
\Delta(e_n)=\sum_{i=0}^n e_i\tensor e_{n-i},
$$
where $e_r$ is the $r$th elementary symmetric polynomial, which is equal to $q^{-\binom{r}2}\psi([\mathcal I_{(1^r)}])$. Note
also that $\psi(\sum_{\lambda\vdash n}[\mathcal I_\lambda])$ is the $n$th complete symmetric function~$h_n$.

Since~$C_{\cat A}=\QQ[\mathcal P_1]$, we have $C_{\cat A}\subsetneq H_{\cat A}$. Since $\dim_\QQ\Prim(H_{\cat A})_n=1$ for all~$n>0$,
$U_{\cat A}=H_{\cat A}$. Thus, $C_{\cat A}\subsetneq E_{\cat A}=U_{\cat A}=H_{\cat A}$.

\subsection{Homogeneous tubes}\label{subs:ex-tube}
Let~$\cat A$ be the category of finite dimensional $\kk$-representations of a 
tame acyclic quiver~$Q$. Then
$\cat A$ decomposes into a triple of subcategories of preprojective, 
preinjective and regular representations (see~\cite{ARS}*{Chapter~VIII}) which 
we denote,
as in~\cite{BG}*{Section~5},
by $\cat A_-$, $\cat A_+$ and~$\cat A_0$, respectively. The category~$\cat A_0$
can be further decomposed into the so-called stable tubes, that is, components 
of the Auslander-Reiten quiver of~$\cat A$ 
on which the Auslander translation acts as an autoequivalence of finite order, 
called the {\em rank}
of the tube. 
It is well-known that rank~$1$, or homogeneous, tubes are parametrized
by the set $\kk\mathbb P^1$ of homogeneous prime ideals in~$\kk[x,y]$. Given a
homogeneous prime ideal~$\rho$,
let $\deg\rho$ be the degree of a generator of that ideal and denote by~$\cat 
T_\rho$ the corresponding rank~$1$ tube. 
Then $\cat T_\rho$ is equivalent to the category of nilpotent representations 
of~$\KK[x]$ where $[\KK:\kk]=\deg\rho$
and its Hall algebra is isomorphic to the classical Hall-Steinitz algebra. Thus, 
for each $r>0$, $\cat T_\rho$ contains 
a unique indecomposable $\mathcal I_r(\rho)$ of length~$r$. Given a 
partition~$\lambda=(\lambda_1\ge
\cdots\ge \lambda_k>0)$, let $\mathcal I_\lambda(\rho)=\mathcal I_{\lambda_1}(\rho)\oplus 
\cdots\oplus\mathcal I_{\lambda_k}(\rho)$.
By~\S\ref{subs:ex-jordan} the elements
$$
\mathcal P_n(\rho)=\sum_{\lambda\vdash n} 
\Big(\prod_{j=1}^{\ell(\lambda)-1}(1-q^{j\deg\rho})\Big) [\mathcal I_\lambda(\rho)]
$$
are primitive in $H_{\cat T_\rho}$. Let~$\cat A_{0,h}$ be the full subcategory 
of homogeneous objects in~$\cat A_0$ (cf.~\cite{DR}*{Theorem~3.5}).
Since $H_{\cat A_{0,h}}$ is isomorphic to the 
tensor product of the~$H_{\cat T_\rho}$  as 
a bialgebra, this gives all primitive elements in~$H_{\cat A_{0,h}}$. The 
Grothendieck monoid of~$\cat A_{0,h}$ equals
the direct sum of infinitely many copies (indexed by $\rho\in\kk\mathbb P^1$) 
of~$\mathbb Z_{\ge 0}$.

However, the elements $\mathcal P_n(\rho)$ are not primitive in~$H_{\cat A}$ 
since an object in~$\cat A_{0,h}$ can have preprojective subobjects
and preinjective quotients. They can be used to construct primitive elements 
in~$H_{\cat A}$.
\begin{conjecture}\label{eq:prims}
The elements $$\mathcal P_n(\rho)-\frac{1}{N(\deg \rho)}\sum_{\rho'\in \kk\mathbb P^1\,:\,\deg 
\rho'=\deg\rho} \mathcal P_n(\rho'),
$$
where $N(d)$ is the number of elements of~$\kk\mathbb P^1$ of degree~$d$, that is,
$N(1)=|\kk|+1$ while $N(d)$, $d>1$, is the number of irreducible monic polynomials of degree~$d$ in one variable,
are primitive in~$H_{\cat A}$.
\end{conjecture}
This formula can be easily checked in small cases (see an example below) or for 
the Kronecker quiver, using the results of~\cite{Sz}.
Since~$F^{I,P}_M=0$ for all $P\in\cat A_-$, $I\in\cat A_+$, the above Corollary is an immediate consequence of the following
\begin{conjecture}
Let~$I\in\cat A_+$, $P\in\cat A_-$.
Then for any partition~$\lambda$ we have $F^{P,I}_{\mathcal I_\lambda(\rho)}=F^{P,I}_{\mathcal I_\lambda(\rho')}$ 
where $\rho,\rho'\in \kk\mathbb P^1$ with
$\deg\rho=\deg\rho'$.
\end{conjecture}
This is known to hold in some special cases (see for example~\cites{Sz,Hub2}).

%Let~$\delta=|\mathcal I_1(\rho)|\in\Gamma_{\cat A}$ where~$\rho\in\kk\mathbb P^1$ with~$\deg\rho=1$. 
%Then Theorem~\ref{thm:Kac-Hua-Haus} implies that~$m_{n\delta}$ is the number of vertices in~$Q$ minus one.
%The dimension of~$\Prim(H_{\cat A})_{n\delta}$ was computed in~\cite{HX}*{Corollary~4.5.1}, 
%which in particular verifies Conjecture~\ref{conj:im-ineq} for~$\gamma=n\delta$. It is worth
%noting that, in view of Proposition~\ref{prop:PBW-prec}, Corollary~4.5.1 of~\cite{HX} implies
%that $m_{n\delta}$, $n>0$ is 
%the number of vertices in the quiver~$Q$ minus one.

In the category~$\cat A$, we have $C_{\cat A}=E_{\cat A}=U_{\cat A}\subsetneq 
H_{\cat A}$. On the other hand,
for~$\cat A_0$ we have $C_{\cat A_0}\subsetneq E_{\cat A_0}=U_{\cat 
A_0}=H_{\cat A_0}$ and similarly for each homogeneous tube.

\subsection{A tame valued quiver}\label{subs:ex-non-simply-laced}
Consider now the valued quiver $1\xrightarrow{(1,4)} 2$. Let~$\kk_2$ be a field 
extension of~$\kk_1=\kk$ of degree~$4$. 
Note that $\kk_2$ contains precisely $q^4-q^2$ elements of degree~$4$ 
over~$\kk$ and $q^2-q$ elements of degree~$2$.
A representation of this quiver is a triple $(V_1,V_2,f)$ where $V_i$ is a 
$\kk_i$-vector space
and $f\in\Hom_{\kk}(V_1,V_2)$. Finally, a morphism $(V_1,V_2,f)\to (W_1,W_2,g)$ 
is a pair $(\varphi_1,\varphi_2)$
where $\varphi_i\in\Hom_{\kk_i}(V_i,W_i)$ and $g\circ \varphi_1=\varphi_2\circ 
g$. 

The smallest regular representation is $(\kk_1^2,\kk_2,f)$, where $f$ is 
injective. Thus, $f$ is given by a pair 
$(\lambda,\mu)\in\kk_2\times\kk_2$ which is linearly independent over~$\kk$ 
(this pair is the image under~$f$ of 
the standard basis of~$\kk_1^2$). It is easy to see that, up to an isomorphism,
such a pair can be assumed to be of the form $(\lambda,1)$ 
where~$\lambda\in\kk_2\setminus\kk_1$. Denote the resulting
representation by $E_1(\lambda)$. If we have a morphism $f:E_1(\lambda)\to 
E_1(\lambda')$, then this morphism 
is uniquely determined by a matrix $\varphi_1=\left(\begin{smallmatrix}
                                          a&b\\c&d
                                         \end{smallmatrix}\right)\in\Mat{2}$ 
and $\varphi_2\in\kk_2$ and we have
$$
(b\lambda'+c)\lambda=a\lambda'+c.
$$

If~$\lambda$ has degree~$4$ over~$\kk$ then $\End_{\cat A}E_1(\lambda)\cong \kk$
and $\Aut_{\cat A}E_1(\lambda)\cong \kk^\times$. Otherwise, $\End_{\cat 
A}E_1(\lambda)\cong L$
and $\Aut_{\cat A}E_1(\lambda)\cong L^\times$ where $[L:\kk]=2$. It follows 
that all $E_1(\lambda)$
with $\deg_\kk\lambda=2$ are isomorphic, since the stabilizer of such a 
$\lambda$ in~$\GL(2,\kk)$ has 
index $q^2-q$, and that there are $q$ non-isomorphic representations 
$E_1(\lambda)$ with~$\deg_\kk\lambda=4$.
It is easy to see that for any~$\lambda\in \kk_2\setminus\kk_1$ we have 
$(q^2-1)(q-1)$ short exact sequences
$$
0\to S_2\to E_1(\lambda)\to P_1\to 0.
$$
As a result, we conclude that 
$$
\bar\Delta(E_1(\lambda))=(q^{\frac12\deg_\kk\lambda}-1) [S_2]\tensor [P_1],
$$
hence
$$
P_1(\lambda):=E_1(\lambda)-\frac{1}{(q+1)|\Aut_{\cat 
A}E_1(\lambda)|}\sum_{\mu\in (\kk_2\setminus\kk_1)/\GL(2,\kk)}|\Aut_{\cat 
A}(E_1(\mu))| E_1(\mu)
$$
is primitive, and these are all primitive elements of 
degree~$\alpha_1+\alpha_2$ in~$H_{\cat A}$. There is precisely one linear 
relation among them, namely 
$\sum_{\lambda\in(\kk_2\setminus\kk_1)/\GL(2,\kk)} |\Aut_{\cat 
A}E_1(\lambda)|P_1(\lambda)=0$.

In this case, like in~\S\ref{subs:ex-tube}, $C_{\cat A}=E_{\cat A}=U_{\cat 
A}\subsetneq H_{\cat A}$ which supports Conjecture~\ref{conj:Lie}.
Also, $\dim_\QQ \Prim(H_{\cat A})_{\alpha_1+\alpha_2}=q$ and $m_{\alpha_1+\alpha_2}=1$.

\subsection{Hereditary categories defined by
submonoids}\label{subs:ex-1,2}\label{subs:equidim}
The next two examples are special cases of the following construction. Let 
$\Gamma_0$ be a submonoid of 
the Grothendieck monoid~$\Gamma$ of an abelian category~$\cat A$ and define a 
full subcategory~$\cat A(\Gamma_0)$
of~$\cat A$ whose objects~$X$ satisfy~$|X|\in\Gamma_0$. By construction, $\cat 
A(\Gamma_0)$ is closed under extensions and
hence is exact.

First, let~$\cat A$ be the category of $\kk$-representations of the quiver 
$1\xrightarrow{} 2$. Then~$\Gamma_{\cat A}$
is freely generated by $\alpha_i=|S_i|$ where $S_i$, $i=1,2$ are simple 
objects. Fix~$r>0$.  
Let~$\Gamma_r=\{ s\alpha_1+r s\alpha_2\,:\,s\in\mathbb Z_{\ge 0}\}$ and set 
${\cat B_r}=\cat A(\Gamma_r)$.
Let $P_1=I_2$ be the projective cover of~$S_1$ and the injective envelope 
of~$S_2$ in~$\cat A$.
Then in~$H_{\cat A}$ we have 
\begin{equation}\label{eq:mult-table-A2}
[S_1][S_2]=[S_2][S_1]+[P_1],\qquad [S_1][P_1]=q [P_1][S_1],\qquad [P_1][S_2]=q 
[S_2][P_1].
\end{equation}
Every object in~${\cat B_r}$ is isomorphic to $S_1^{\oplus a}\oplus P_1^{\oplus 
b}\oplus S_2^{\oplus (ra+(r-1)b)}$,
$a,b\ge 0$. The only simple objects in~$\cat B_r$, up to an isomorphism,  are $X_1=S_1\oplus S_2^{\oplus r}$
and $X_2=S_2^{\oplus r-1}\oplus P_1$. Then $[X_i]$ is a non-zero multiple of~$E_i$,
$i=1,2$, where $E_1=[S_2]^r[S_1]$ and~$E_2=[S_2]^{r-1}[P_1]$. 
In particular, the $E_i$ are primitive elements of~$H_{{\cat B_r}}$.
Using~\eqref{eq:mult-table-A2} we can show that $E_1$, $E_2$ satisfy the 
relation
$$
E_2 E_1=q^{r-1} E_1 E_2-[r-1]_q E_2^2,
$$
where $[s]_q=1+\cdots+q^{s-1}$.
The Grothendieck monoid of~${\cat B_r}$ is generated by $\beta_i=|X_i|$, $i=1,2$,
subject to the relation $\beta_1+\beta_2=2\beta_1=2\beta_2$ (thus $\Gamma_{{\cat B_r}}$ 
does not coincide with~$\Gamma_r$ and is not even a submonoid of~$\Gamma_{\cat 
A}$).
It is not hard to check that $E_1$ and~$E_2$ generate~$H_{\cat B_r}$ and hence form a basis of~$\Prim(H_{\cat B_r})$.

In this case we have $C_{{\cat B_r}}=U_{{\cat B_r}}=E_{{\cat B_r}}=H_{{\cat B_r}}$ and so 
Conjecture~\ref{conj:Lie} holds.

A more complicated example is obtained as follows.
Let $\cat A$ be the category of $\kk$-representations of the quiver $1\to 
0\xleftarrow{} 2$. 
As in the previous example $\Gamma_{\cat A}$ is freely generated by 
$\alpha_i=|S_i|$, $0\le i\le 2$.
Let~$\Gamma_\circ=\{ s\alpha_0+r\alpha_1+r\alpha_2\,:\, r,s\in\mathbb Z_{\ge 0}\}$ 
and let~$\cat B=\cat A(\Gamma_\circ)$.
Let~$P_i$ be the projective cover of~$S_i$ in~$\cat A$ and~$I_i$ be its 
injective envelope.
Thus, $I_1=S_1$, $I_2=S_2$, $|I_0|=\alpha_0+\alpha_1+\alpha_2$, 
$|P_1|=\alpha_0+\alpha_1$, $P_0=S_0$ and $|P_2|=\alpha_2+\alpha_0$.
The simple objects in~$\cat B$ are
$S_1\oplus S_2$ and~$S_0$, while
the non-simple indecomposable objects are
$$
P_1\oplus S_2,\quad P_2\oplus S_1,\quad  P_1\oplus P_2,\qquad I_0.
$$
The Grothendieck monoid of~$\cat B$ is freely generated by $\beta_1=|S_1\oplus 
S_2|$ and~$\beta_0=|S_0|$.
Clearly, $Y_1=[S_1\oplus S_2]$ and~$Y_0=[S_0]$ are primitive in~$H_{\cat B}$. 
We also have two linearly independent 
primitive elements of degree~$\beta_1+\beta_0$, say $Z_1=[I_0]-(q-1)[P_1\oplus 
S_2]$ and $Z_2=[I_0]-(q-1)[P_2\oplus S_1]$. 
Then 
\begin{gather*}
[Z_1,Z_2]=0,\qquad [Y_1,Z_1]_q=[Y_1,Z_2]_q=0,\qquad [Z_1,Y_0]_q=[Z_2,Y_0]_q=0,\\
\intertext{and}
[Y_1,[Y_1,Y_0]]_{q^2}=Y_1(Z_1+Z_2),\qquad
[[Y_1,Y_0],Y_0]_q,Y_0]_{q^2}=0,
\end{gather*}
where~$[a,b]_t=ab-t ba$.
Here $C_{\cat B}=E_{\cat B}=U_{\cat B}\subsetneq H_{\cat B}$ which again 
supports Conjecture~\ref{conj:Lie}.
Also, we have a unique imaginary simple root~$\beta_1+\beta_0$, and $\dim_\QQ 
\Prim(H_{\cat B})_{\beta_1+\beta_0}=2$ and $m_{\beta_1+\beta_0}=1$.

\subsection{Sheaves on projective curves}\label{subs:ex-sheaves}
Consider the category~$\cat A$ of coherent sheaves on~$\mathbb P^1(\kk)$ (cf. 
\cites{BS,Kap,BK}). Following~\cite{BK},
$\cat A$ is equivalent to the category
whose objects are triples $(M',M'',\phi)$ where $M'$ is a $\kk[z]$-module,
$M''$ is a $\kk[z^{-1}]$-module and $\phi$ is an isomorphism of 
$\kk[z,z^{-1}]$-modules $M'_z\to M''_{z^{-1}}$.
In particular, for any $n\in\mathbb Z$, we have an indecomposable object 
$\mathcal O(n)=(\kk[z],\kk[z^{-1}],\phi_n)$
where $\phi_n\in\Aut\kk[z,z^{-1}]$ is the multiplication by~$z^{-n}$.
 We have (cf.~\cite{BK})
$$
\dim_\kk\Hom_{\cat A}(\mathcal O(m),\mathcal O(n))=\max(0,n-m+1)
$$
and any non-zero morphism $\mathcal O(m)\to\mathcal O(n)$ is injective.
Consider now the full subcategory~$\cat A_{lc}$ of locally free coherent 
sheaves on~$\mathbb P^1$. Any object in~$\cat A_{lc}$
is isomorphic to a direct sum of objects of the form $\mathcal O(m)$ and these 
are precisely the indecomposables in~$\cat A_{lc}$.
The Grothendieck monoid of~$\cat A_{lc}$ identifies with $\mathbb Z_{\ge 
0}\times\mathbb Z$ with $|\mathcal O(n)|=(1,n)$.
Note that $\cat A_{lc}$ has no simple objects. 
The category~$\cat A_{lc}$ is closed under extensions and hence is exact. Since
it is Krull-Schmidt, its Hall algebra has a basis consisting of ordered 
monomials on~$X_m:=[\mathcal O(m)]$ for {\em any} total order on~$\mathbb Z$.
Since $m<n$ implies that 
$\mathcal O(n)/\mathcal O(m)$ is not an object in~$\cat A_{lc}$, it follows 
that $\mathcal O(m)$ is almost simple, hence
$X_m$ is primitive for all~$m\in\mathbb Z$. Thus, $H_{\cat A_{lc}}$
is primitively generated.
By~\cite{BK}*{Theorem~10(iii)} the defining relations in~$H_{\cat A_{lc}}$ are
$$
X_n X_m=q^{n-m+1}X_m X_n+(q^2-1)q^{n-m-1}\sum_{a=1}^{\lfloor (n-m)/2\rfloor} 
X_{m+a}X_{n-a},\qquad m<n.
$$

However, Theorem~\ref{th:prim generation} does not apply to the Hall algebra 
of~$\cat A$ or even~$\cat A_{lc}$ since 
$\cat A$ is neither cofinitary nor profinitary. Indeed, every object $\mathcal 
O(m)\oplus \mathcal O(n)$, $m>n$ appears as the middle term 
of a short exact sequence 
$$
0\to \mathcal O(n-a)\to \mathcal O(m)\oplus\mathcal O(n)\to \mathcal O(m+a)\to 0
$$
for all $a\ge 0$. 

On the other hand, the Hall algebra of the subcategory of torsion sheaves
is isomorphic to the Hall algebra of the regular subcategory for the valued 
quiver $1\xrightarrow{(2,2)} 2$, or, equivalently, the Kronecker quiver.

It should be noted that the Hall algebra of the subcategory of preprojective 
modules~$\cat B_+$ in the 
category~$\cat B$ of $\kk$-representations of Kronecker quiver 
is isomorphic to the subalgebra of~$H_{\cat A_{lc}}$ generated by the~$X_m$, 
$m\ge 0$.
Indeed, $\Gamma_{\cat B_+}\cong \mathbb Z_{\ge 0}$, and for each~$k\ge 0$
there is a unique preprojective indecomposable~$Q_k$ with 
$|Q_k|=k$. It is easy to see, 
by grading considerations, that~$Q_k$ is primitive. Then the $[Q_k]$, $k\ge 0$ 
can be shown to satisfy exactly 
the same relations as the $X_n$ (see~\cite{Sz}*{Theorem~4.2}).
In this case we have $C_{\cat B_+}\subsetneq U_{\cat B_+}=E_{\cat B_+}=H_{\cat 
B_+}$.

This situation can be generalized as follows. Let~$X$ be a smooth projective 
curve and let~$\cat A$ be the category
of coherent sheaves on~$X$. Let~$\cat A^{\ge d}_{lc}$ be the full subcategory 
of~$\cat A$ whose objects 
are locally free sheaves of positive rank and of degree~$\ge d$. 
Since the rank 
and the degree are additive on
short exact sequences, this subcategory is closed under extensions. 
Since (cf.~\cite{KSV}*{Proposition~2.5})
for a coherent sheaf~$\mathcal F$ the possible degrees of its subsheaves of 
rank~$r$ are bounded above,
for any fixed pair $(r,d)$ there are finitely many subsheaves of~$\mathcal F$ 
of rank~$r$ and degree~$d$.
We conclude that the category $\cat A^{\ge d}_{lc}$ is cofinitary and 
profinitary, hence Theorem~\ref{th:prim generation} applies and the Hall 
algebra 
of~$\cat A^{\ge d}_{lc}$ is generated by its primitive elements. Results on 
primitive elements in this algebra can be found in~\cite{KSV}*{\S3.2}. Note that 
$\cat A_{lc}$ is Krull-Schmidt, hence its Hall algebra is PBW on indecomposables.

\subsection{Non-hereditary categories of finite type}\label{subs:ex-square}

Let $\cat A$ be the category of $\kk$-representations of the quiver
$$
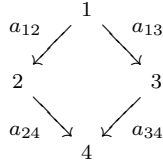
\let\objectstyle\scriptstyle
\xymatrix@=3ex{&1\ar[ld]_{a_{12}}\ar[rd]^{a_{13}}\\2\ar[rd]_{a_{24}}&&3\ar[ld]^{
a_{34}}\\&4}
$$
satisfying the relation $a_{24}a_{12}=0$. This category has 14 isomorphism 
classes of 
indecomposable objects, 12 of them having different images 
in~$\Gamma_{\cat A}$ and
the two remaining ones, namely the projective cover~$P_1$ of~$S_1$ and the 
injective envelope 
$I_4$ of~$S_4$, having the same image $\alpha_1+\alpha_2+\alpha_3+\alpha_4$ (as 
before,
$\alpha_i=|S_i|$).

Let~$S_{ij}$, $S_{ijk}$ be the unique, up to an isomorphism, indecomposables
with $|X|=\alpha_i+\alpha_j$ (respectively, $|X|=\alpha_i+\alpha_j+\alpha_k$).
Then one can easily check that $[S_{ij}],[S_{ijk}]\in P$, hence $\Prim(H_{\cat A})_{\alpha_i+\alpha_j}=
0=\Prim(H_{\cat A})_{\alpha_i+\alpha_j+\alpha_k}$ by Proposition~\ref{prop:PBW-prec}. Let us show
that
$\Prim(H_{\cat A})_{\alpha_1+\alpha_2+\alpha_3+\alpha_4}=0$; then the only primitive elements 
are those in $\Prim(H_{\cat A})_{\alpha_i}$, $1\le i\le 4$.

For every object $M$ with $|M|=\alpha_1+\alpha_2+\alpha_3+\alpha_4$, except 
$P_1$, $I_4$ and $S_2\oplus S_{134}$,
there exists a pair of objects $A$, $B$ such that $F^{A,B}_N=0$ unless 
$[N]=[M]$. This implies that 
$\Prim(H_{\cat A})_{\alpha_1+\alpha_2+\alpha_3+\alpha_4}$ is spanned by 
$[P_1]$, $[I_4]$ and $[S_2\oplus S_{134}]$.
We have (with $h=|\kk^\times|=q-1$)
\begin{align*}
\begin{split}
\bar\Delta([&S_2\oplus S_{134}])=
                    [S_{134} ]\tensor [ S_2 ]+  [S_2 ]\tensor [ S_{134} 
]\\&\qquad\qquad+
                 h( [S_2\oplus S_{34} ]\tensor [ S_1 ]+
                 [S_2\oplus S_4 ]\tensor [ S_{13} ]+
                 [S_{34} ]\tensor [ S_1\oplus S_2 ]+
                  [S_4 ]\tensor [ S_2\oplus S_{13} ])
                  \end{split}\\
\begin{split}
\bar\Delta([&I_4])=                    h([S_{134} ] \tensor [ S_2 ]+
                    [S_{234} ] \tensor [ S_1 ]+
                    [S_{24} ] \tensor [ S_{13}])+
                    h^2([S_{34} ] \tensor [ S_1\oplus S_2 ]+
                    [S_4 ] \tensor [ S_2\oplus S_{13}])
\end{split}\\
\begin{split}
 \bar\Delta([&P_1])=   h([S_{34} ] \tensor [ S_{12} ]+
                    [S_2 ] \tensor [ S_{134} ]+
                    [S_4 ] \tensor [ S_{123}])+            
h^2([S_2\oplus S_{34} ] \tensor [ S_1 ]+
                    [S_2\oplus S_4 ] \tensor [ S_{13}]).
\end{split}
\end{align*}
It is now clear that $\Prim(H_{\cat 
A})_{\alpha_1+\alpha_2+\alpha_3+\alpha_4}=0$.

Let $E_i=[S_i]$, $1\le i\le 4$. 
To write a presentation of $H_{\cat A}$, it is useful to 
introduce~$Z=[P_1]+[I_4]-(q-1)[S_2\oplus S_{134}]$. We obtain
\begin{equation}\label{eq:common}
\begin{split}
&[E_i,[E_i,E_j]]_q=0=[[E_i,E_j],E_j]_q,\quad 
(i,j)\in\{(1,2),(1,3),(2,4),(3,4)\},\\
&[E_2,E_3]=0=[E_1,E_4],
\end{split}
\end{equation}
and also
\begin{gather*}
[E_4,[E_1,E_2]]=0,\qquad [E_1,Z]_q=0=[Z,E_4]_q,\qquad [E_2,Z]=0=[E_3,Z]
\end{gather*}
where
$$
Z=[E_1,[E_2,[E_3,E_4]]_q]-[E_4,[E_3,[E_2,E_1]]_q].
$$

If we consider the category of representations of the same quiver satisfying 
the relation $a_{24}a_{12}=a_{34}a_{13}$,
its Hall algebra's subspace of primitive elements is spanned by the $E_i$, 
$1\le i\le 4$ which satisfy~\eqref{eq:common},
as well as 
\begin{gather*}
[E_4,[E_1,E_2]]=0=[E_4,[E_1,E_3]]\\
[E_1,[E_2,[E_3,E_4]]]=[E_4,[E_3,[E_1,E_2]]]=[E_4,[E_2,[E_1,E_3]]]=[E_1,[E_3,[E_2
,E_4]]].
\end{gather*}

In both cases $C_{\cat A}=E_{\cat A}=U_{\cat A}=H_{\cat A}$.

\subsection{Special pairs of objects and primitive elements}\label{sec:grp-ex}
Before we consider the next group of examples, we make the following 
observation.
Suppose that we have a pair of indecomposable objects $X\not\cong Y$ in~$\cat 
A$ satisfying $\Hom_{\cat A}(X,Y)=0=\Hom_{\cat A}(Y,X)$,
$\End_{\cat A}X\cong\End_{\cat A}Y\cong \kk$ is a field
and $\dim_\kk\Ext^1_{\cat A}(X,Y)=\dim_\kk\Ext^1_{\cat A}(Y,X)=1$. Then there 
exist
unique $[Z_{YX}],[Z_{XY}]\in\Iso\cat A$ such that $\{[Z_{YX}]\}=\uExt^1_{\cat 
A}(Y,X)$ 
and $\{[Z_{XY}]\}=\uExt^1_{\cat A}(X,Y)$.
Let~$\cat B=\cat A(X,Y)$ be the minimal full subcategory of~$\cat A$ containing 
$X$ and~$Y$ and closed under extensions.
Then in~$H_{\cat B}$ we have
$$
\bar\Delta([Z_{YX}])=(q-1)[X]\tensor [Y],\quad 
\bar\Delta([Z_{XY}])=(q-1)[Y]\tensor [X],\quad \bar\Delta([X\oplus 
Y])=[X]\tensor [Y]+[Y]\tensor [X],
$$
and so
$$
[Z_{XY}]+[Z_{YX}]-(q-1)[X\oplus Y]
$$
is primitive in~$H_{\cat B}$. Indeed, $|\Ext_{\cat 
A}^1(Y,X)_{Z_{YX}}|=q-1=|\Ext^1_{\cat A}(X,Y)_{Z_{XY}}|$ and so by Riedtmann's 
formula
$$
F^{X,Y}_{Z_{YX}}=q-1=F^{Y,X}_{Z_{XY}},\qquad F^{X,Y}_{X\oplus 
Y}=F^{Y,X}_{X\oplus Y}=1.
$$
This element need not be primitive in~$H_{\cat A}$ but is often useful for 
computations.

\subsection{A rank~2 tube}
Let~$\cat A=\Rep_\kk(Q)$ where $Q$ is a valued acyclic quiver of tame type.
Let~$\tau$ be the Auslander-Reiten translation and consider a regular component 
of the Auslander-Reiten quiver
which is a tube of rank~$2$ (that is, for every indecomposable object~$M$ in 
that component we have
$\tau^2(M)\cong M$). 
The smallest example is provided by
the quiver
$$\let\objectstyle\scriptstyle
\xymatrix@C=1ex@R=2ex{&2\ar[rd]\\
1\ar[ru]\ar[rr]&&3}
$$
and the Auslander-Reiten component containing~$S_2$.

Let $X$ be a simple object in our tube. Then $\tau(X)$ is also simple and both
satisfy $\End_{\cat A}X\cong\End_{\cat A}\tau(X)\cong\kk$. Furthermore,
$$
\Ext^1_{\cat A}(X,\tau(X))\cong\Hom_{\cat A}(\tau(X),\tau(X)),\qquad 
\Ext^1_{\cat A}(\tau(X),X)\cong\Hom_{\cat A}(X,X),
$$
and so $X$, $\tau(X)$ satisfy the assumptions of~\S\ref{sec:grp-ex}. Thus, we 
obtain a 
primitive element of degree~$|X|+|\tau(X)|$ in the Hall algebra of our tube 
given by $Z_{X,\tau(X)}+Z_{\tau(X),X}-(q-1)[X\oplus Y]$.
For the quiver shown above, with $X=S_2$ we have
$$
Y=\tau(X)=\vcenter{\xymatrix@C=1ex@R=2ex{&0\ar[rd]^0\\
\kk\ar[ru]^0\ar[rr]^1&&\kk}}
$$ 
while
$$
Z_{YX}=\vcenter{\xymatrix@C=1ex@R=2ex{&\kk\ar[rd]^0\\
\kk\ar[ru]^1\ar[rr]^1&&\kk}},\qquad 
Z_{XY}=\vcenter{\xymatrix@C=1ex@R=2ex{&\kk\ar[rd]^1\\
\kk\ar[ru]^0\ar[rr]^1&&\kk}}.
$$
However, in~$H_{\cat A}$ we have 
$$
\bar\Delta_{\cat A}(Z_{YX}+Z_{XY}-(q-1)[X\oplus Y])=(q-1) ([S_3]\tensor 
[I_2]+[P_2]\tensor [S_1])
$$
Other indecomposable objects with the same image in~$\Gamma_{\cat A}$ are, up 
to an isomorphism
$$
E_1(\lambda)=\vcenter{\xymatrix@C=1ex@R=2ex{&\kk\ar[rd]^1\\
\kk\ar[ru]^1\ar[rr]^\lambda&&\kk}},\qquad \lambda\in\kk.
$$
and we have
$$
\bar\Delta(E_1(\lambda))=(q-1)([S_3]\tensor [I_2]+[P_2]\tensor [S_1])
$$
where~$I_2$ is the injective envelope of~$S_2$ and $P_2$ is its projective 
cover. This gives $q-1$ linearly independent
primitive elements 
$$
\mathcal P_1(\lambda)=E_1(\lambda)-\frac{1}{q}\sum_{\mu\in\kk} E_1(\mu)
$$
and one more primitive element
$$
[Z_{YX}]+[Z_{XY}]-(q-1)[X\oplus Y]-\frac1 q\sum_{\lambda\in\kk} E_1(\lambda).
$$
Thus, in this case $m_{\alpha_1+\alpha_2+\alpha_3}=2$ and
$
\dim \Prim(H_{\cat A})_{\alpha_1+\alpha_2+\alpha_3}=q%=\#\Ind\cat 
%A_{\alpha_1+\alpha_2+\alpha_3}-2.
$.

In general, primitive elements in Hall algebras corresponding to non-homogeneous tubes were computed in~\cite{Hub1}.
It should be noted that they are not primitive in~$H_{\cat A}$ but, conjecturally, can be used to 
construct primitive elements in a way similar to shown above.

\subsection{Cyclic quivers with relations}
Let $\cat A$ be the category of representations of the quiver 
$$
\let\objectstyle\scriptstyle
\xymatrix{1\ar@<+.5ex>[r]^{a_{12}} & 2\ar@<+.5ex>[l]^{a_{21}}}
$$
satisfying the relation $a_{21}a_{12}=0$. The three non-simple indecomposable 
objects are, up to an isomorphism 
$$
S_{12}:\xymatrix{\kk\ar@<+.5ex>[r]^{1} & \kk\ar@<+.5ex>[l]^{0}}
\qquad S_{21}:\xymatrix{\kk\ar@<+.5ex>[r]^{0} & \kk\ar@<+.5ex>[l]^{1}}\qquad 
S_{212}: \xymatrix{\kk\ar@<+.5ex>[r]^{\binom{1}{0}} & 
\kk^2\ar@<+.5ex>[l]^{(0\,1)}}.
$$
The object $S_{12}$ is the projective cover of~$S_1$ while~$S_{21}$ is its 
injective envelope. Thus,
$$
\bar\Delta([S_{12}])=(q-1)[S_2]\tensor [S_1],\qquad \bar\Delta([S_{21}])=(q-1) 
[S_1]\tensor [S_2]
$$
and so 
\begin{equation}\label{eq:prim-Z}
Z=[S_{12}]+[S_{21}]-(q-1)[ S_1\oplus S_2]
\end{equation}
is the unique, up to a scalar, primitive element, apart from the simple ones. 
Let $E_1=[S_1]$, $E_2=[S_2]$. Then 
$$
[E_1,Z]=[E_2,Z]=0
$$
and 
$$
[E_1,[E_1,E_2]_q]_{q^{-1}}=(1-q^{-1}) E_1 Z,\qquad 
[E_2,[E_2,[E_2,E_1]]_q]_{q^{-1}}=0
$$
is a presentation of~$H_{\cat A}$.

Now let~$\cat A$ be the category of representations of the same quiver 
satisfying the relations
$a_{21}a_{12}=0=a_{12}a_{21}$. In this case, we have four indecomposable 
objects $S_1$, $S_2$, $S_{12}$ and~$S_{21}$ which 
coincide with the ones listed above.
Thus, $S_{ij}$ is the injective envelope of~$S_i$ and the projective cover 
of~$S_j$, $\{i,j\}=\{1,2\}$.
As before, we have a unique non-simple primitive element given by the same 
formula~\eqref{eq:prim-Z}. The following 
provides a presentation for~$H_{\cat A}$
$$
[E_1,[E_1,E_2]_q]_{q^{-1}}=(1-q^{-1})E_1 Z,\qquad 
[E_2,[E_2,E_1]_q]_{q^{-1}}=(1-q^{-1})E_2 Z,\qquad [E_1,Z]=[E_2,Z]=0.
$$

In both examples, we have $C_{\cat A}\subsetneq U_{\cat A}=E_{\cat A}=H_{\cat 
A}$ which contributes supporting evidence
for Conjecture~\ref{conj:Lie}. Note also that in this case~$m_\gamma>0$
%\dim_\QQ \Prim(H_{\cat A})_\gamma<\#\Ind\cat A_\gamma$ 
for~$\gamma$ imaginary.

\section{PBW property of Hall algebras and proof of Theorem~\ref{thm:PBW-prop-Hall}}\label{sec:PBW}

\subsection{Rings filtered and graded by ordered monoids}\label{subs:filt-gr-ring}
Let~$(\Lambda,\lhd)$ be an ordered abelian monoid, as defined in~\S\ref{subs:ord-mon-PBW-prop}.
We write~$\mu\unlhd\nu$ if either $\mu=\nu$ or $\mu\not=\nu$ and $\mu\lhd\nu$. 
\begin{definition}\label{def:filtered-ring}
We say that a unital ring $\mathcal H$ is $\Lambda$-filtered if 
$\mathcal H$ contains a family of abelian subgroups $\mathcal H^{\unlhd\lambda}$, $\lambda\in\Lambda$, such that
\begin{enumerate}[(i)]
\item\label{def:filtered-ring.0} $1_{\mathcal H}\in\mathcal H^0:=\mathcal H^{\unlhd 0}$ for all~$\lambda\in \Lambda$;
 \item\label{def:filtered-ring.i} $\lambda\unlhd\mu\implies \mathcal H^{\unlhd\lambda}\subset \mathcal H^{\unlhd\mu}$;
 \item\label{def:filtered-ring.ii} $\mathcal H=\bigcup_{\lambda\in\Lambda}\mathcal H^{\unlhd \lambda}$;
 \item\label{def:filtered-ring.iv} 
$\mathcal H^{\unlhd \lambda}\cdot\mathcal H^{\unlhd \mu}
\subset \mathcal H^{\unlhd (\lambda+\mu)}$ for all $\lambda,\mu\in\Lambda$.
\end{enumerate}
\end{definition}
\noindent
In particular, $\mathcal H^0\subset\bigcap_{\lambda\in\Lambda}\mathcal H^{\unlhd \lambda}$ and is a subring of~$\mathcal H$.
This definition is similar to that in~\cite{P-P}*{\S4.7}; however, we do not require the ring~$\mathcal H$ to admit 
a $\mathbb Z_{\ge 0}$-grading compatible with~$\Lambda$.

Given~$\lambda\in\Lambda$, let $\mathcal H^{\lhd\lambda}=\sum_{\mu\lhd\lambda} \mathcal H^{\unlhd \mu}$. Note that 
\begin{equation}\label{eq:leq}
\mathcal H^{\lhd\lambda}\cdot\mathcal H^{\unlhd\mu}\subset 
\mathcal H^{\lhd(\lambda+\mu)},\qquad \mathcal H^{\unlhd\lambda}\cdot\mathcal H^{\lhd\mu}\subset 
\mathcal H^{\lhd(\lambda+\mu)}.
\end{equation}
Define
$$
\gr_\Lambda \mathcal H=\bigoplus_{\lambda\in \Lambda} \mathcal H^{\unlhd \lambda}/\mathcal H^{\lhd\lambda}.
$$
\begin{lemma}
$\gr_\Lambda \mathcal H$ is a $\Lambda$-graded unital ring with the multiplication given by
$$
(x+\mathcal H^{\lhd \lambda})\bullet(y+\mathcal H^{\lhd\mu})=x\cdot y+\mathcal H^{\lhd(\mu+\nu)},\qquad 
x\in\mathcal H^{\unlhd\lambda},\, y\in\mathcal H^{\unlhd \mu}.
$$
\end{lemma}
\begin{proof}
By construction, $\gr_\Lambda\mathcal H$ is an abelian group.
Using~\eqref{eq:leq}, we obtain, for all $x\in\mathcal H^{\unlhd\lambda}$, $y\in\mathcal H^{\unlhd\lambda}$ 
$$
(x+\mathcal H^{\lhd\lambda})\cdot (y+\mathcal H^{\lhd\mu}) \subset x\cdot y+
\mathcal H^{\unlhd \lambda}\cdot \mathcal H^{\lhd\mu}+
\mathcal H^{\lhd\lambda}\cdot\mathcal H^{\unlhd \mu}+\mathcal H^{\lhd\lambda}\cdot\mathcal H^{\lhd\mu}
\subset x\cdot y+\mathcal H^{\lhd(\lambda+\mu)}.
$$
Thus, $\,\bullet\,$ is well-defined. The distributivity and the associativity follow from those in~$\mathcal H$. 
Then the ring~$\mathcal H$ is graded by~$\Lambda$ by construction. Finally, observe that $1_{\mathcal H}
+\mathcal H^{\lhd\lambda}=0$
for all~$\lambda\in\Lambda^+$ and it is easy to see that~$1_{\mathcal H}+\mathcal H^{\lhd 0}$ is the unity of~$\gr_\Lambda\mathcal H$.
\end{proof}
Let~$\bar{\mathcal H}^\lambda=\mathcal H^{\unlhd\lambda}/\mathcal H^{\lhd \lambda}$. Note that~$\bar{\mathcal H}^0$ identifies with~$\mathcal H^0$,
since~$\mathcal H^{\lhd 0}=\{0\}$.

Let~$\Lambda_{\min}$ be the set of minimal, with respect to the partial order~$\unlhd$, elements of~$\Lambda^+$. We say 
that $\Lambda$ is optimal if it is generated by~$\Lambda_{\min}$. 
\begin{lemma}\label{lem:filtered-graded}
Let $(\Lambda,\unlhd)$ be an optimal monoid and let~$\mathcal H$ be a $\Lambda$-filtered ring. 
Then for any subset $\Lambda_{\circ}$ of~$\Lambda_{\min}$ which generates~$\Lambda$ as a monoid
the following are equivalent
\begin{enumerate}[{\rm(i)}]
 \item\label{lem:filtered-graded.a} $\mathcal H$ is generated over~$\mathcal H^0$ by the $\mathcal H^{\unlhd \lambda}$,
 $\lambda\in\Lambda_\circ$;
\item\label{lem:filtered-graded.b} $\gr_\Lambda \mathcal H$ is generated over~$\mathcal H^0=\bar{\mathcal H^0}$
by the~$\bar{\mathcal H}^{\lambda}$, $\lambda\in\Lambda_\circ$. 
\end{enumerate}
\end{lemma}
\begin{proof}
Let~$\mathcal H_{\circ}=\sum_{\lambda\in\Lambda_\circ}\mathcal H^{\unlhd \lambda}$ and $\bar{\mathcal H}_{\circ}
=\bigoplus_{\lambda\in\Lambda_\circ} \bar{\mathcal H}^\lambda=\mathcal H_\circ/\mathcal H^0$. Clearly,
the assertion of~\eqref{lem:filtered-graded.a} (respectively, of~\eqref{lem:filtered-graded.b})
means that $\mathcal H_{\circ}^k$ (respectively, $(\bar{\mathcal H}^0\oplus\bar{\mathcal H}_{\circ})^{\bullet k}$) is an increasing filtration on
$\mathcal H$ (respectively, on~$\gr_\Lambda\mathcal H$).
The implication~\eqref{lem:filtered-graded.a}$\implies$\eqref{lem:filtered-graded.b} is immediate. 
For the converse, it suffices to show that for all~$x\in \mathcal H$, 
there exists~$k=k(x)\ge 0$ such that 
$x\in \mathcal H_{\circ}^k$. We argue by
induction on~$(\Lambda^+,\unlhd)$.

The induction base is obvious since for all~$x\in\mathcal H_{\circ}$ we can take~$k=1$. 
Suppose that~$x\in\mathcal H\setminus\mathcal H_{\circ}$. Then~$x\in \mathcal H^{\unlhd\lambda}$ for some~$\lambda\in\Lambda^+\setminus\Lambda_\circ$. 
If~$x\in\mathcal H^{\unlhd\mu}$
for some~$\mu\lhd\lambda$ then we are done by the induction hypothesis. Therefore, we may assume that $x\in\mathcal H^{\unlhd \lambda}\setminus 
\mathcal H^{\lhd\lambda}$ hence $x+\mathcal H^{\lhd\lambda}\not=0$ in~$\gr_\Lambda\mathcal H$. By assumption of~\eqref{lem:filtered-graded.b},
$$
x+\mathcal H^{\lhd\lambda}\in 
\bar{\mathcal H}^\lambda\cap(\bar{\mathcal H}^0\oplus \bar{\mathcal H}_{\circ})^{\bullet k'} 
$$
for some~$k'>0$.
Note that, by definition of~$\gr_\Lambda\mathcal H$
\begin{equation}\label{eq:generators-1}
(\bar{\mathcal H}^0\oplus\bar{\mathcal H}_{\circ})^{\bullet r}\cap \bar{\mathcal H}^\lambda\subset
\sum_{(\lambda_1,\dots,\lambda_{r })\in(\Lambda_\circ\cup\{0\})^{r }\,:\, \lambda_1+\cdots+\lambda_{r }=\lambda} \bar{\mathcal H}^{\lambda_1}
\bullet\cdots\bullet \bar{\mathcal H}^{\lambda_{r }}
\end{equation}
for all~$r>0$, hence 
\begin{equation}\label{eq:generators-2}
(\mathcal H_{\circ})^{ r}\cap \mathcal H^{\unlhd \lambda}\subset
\sum_{(\lambda_1,\dots,\lambda_{r })\in(\Lambda_\circ\cup\{0\})^{r }\,:\, \lambda_1+\cdots+\lambda_{r }=\lambda} (\mathcal H^{\unlhd\lambda_1})
\cdots(\mathcal H^{\unlhd\lambda_{r }})+\mathcal H^{\lhd \lambda}.
\end{equation}
Therefore, $x=x'+x''$ where 
$$
x'\in\sum_{(\lambda_1,\dots,\lambda_{k'})\in(\Lambda_\circ\cup\{0\})^{k'}\,:\,\lambda_1+\cdots+\lambda_{k'}=\lambda} 
(\mathcal H^{\unlhd\lambda_1})\cdots (\mathcal H^{\unlhd\lambda_{k'}})\subset (\mathcal H_{\circ})^{k'},\qquad x''\in\mathcal H^{\lhd\lambda}.
$$
By the induction hypothesis, $x''\in (\mathcal H_{\circ})^{k''}$ for some~$k''>0$, hence~$x\in(\mathcal H_{\circ})^{k}$ where~$k=\max(k',k'')$.
\end{proof}
\begin{proposition}[weak PBW property]\label{prop:filtered-graded}
Let $(\Lambda,\unlhd)$ be an optimal monoid and let~$\mathcal H$ be a $\Lambda$-filtered ring. 
Let $\Lambda_{\circ}\subset\Lambda_{\min}$ be a subset which generates~$\Lambda$ as a monoid.
Suppose that there exists a total order~$\le$ on~$\Lambda_{\circ}$ such that 
$$
\gr_\Lambda\mathcal H=\sum_{k\ge 0}\,\sum_{\lambda_1\le\cdots\le\lambda_k\in\Lambda_{\circ}^k} \bar{\mathcal H}^{\lambda_1}
\bullet\cdots \bullet\bar{\mathcal H}^{\lambda_k}.
$$
Then 
$$
\mathcal H=\sum_{k\ge 0}\,\sum_{\lambda_1\le\cdots\le\lambda_k\in\Lambda_{\circ}^k} \mathcal H^{\unlhd\lambda_1}
\cdots \mathcal H^{\unlhd\lambda_k}.
$$
\end{proposition}
\begin{proof}
Retain the notation of the proof of Lemma~\ref{lem:filtered-graded}. Our assumptions imply that for all~$r>0$ and~$\lambda\in\Lambda^+$,
$$
(\bar{\mathcal H}_{\circ})^{\bullet r}\cap \bar{\mathcal H}^{\lambda}\subset 
\sum_{(\lambda_1\le\cdots\le\lambda_{r })\in\Lambda_\circ^{r }\,:\, \lambda_1+\cdots+\lambda_{r }=\lambda} \bar{\mathcal H}^{\lambda_1}
\bullet\cdots\bullet \bar{\mathcal H}^{\lambda_{r }},
$$
hence
$$
(\bar{\mathcal H}^0\oplus\bar{\mathcal H}_{\circ})^{\bullet r}\cap \bar{\mathcal H}^{\lambda}\subset 
\sum_{s\le r}\,\,\sum_{(\lambda_1\le\cdots\le\lambda_{s })\in\Lambda_\circ^{s }\,:\, \lambda_1+\cdots+\lambda_{s }=\lambda} \bar{\mathcal H}^{\lambda_1}
\bullet\cdots\bullet \bar{\mathcal H}^{\lambda_{s }},
$$
which implies that
$$
( \mathcal H_{\circ})^{ r}\cap \mathcal H^{\unlhd \lambda}\subset
\sum_{s\le r}\,\,\sum_{(\lambda_1\le\cdots\le\lambda_{s })\in\Lambda_\circ^{s }\,:\, \lambda_1+\cdots+\lambda_{s }=\lambda} 
(\mathcal H^{\unlhd\lambda_1})
\cdots(\mathcal H^{\unlhd\lambda_{s }})+\mathcal H^{\lhd \lambda}.
$$
Since~$\mathcal H^{\lhd\lambda}=\sum_{\mu\lhd\lambda}\mathcal H^{\unlhd\mu}$ and 
$(\mathcal H_\circ)^r$, $r\ge 0$ is an increasing filtration on~$\mathcal H$ by Lemma~\ref{lem:filtered-graded}, 
an obvious induction on~$(\Lambda,\unlhd)$ completes the proof.
\end{proof}

We now consider a special case which we will later apply to Hall algebras.
\begin{corollary}\label{cor:pre-PBW}
Let~$(\Lambda,\unlhd)$ be an optimal monoid and let~$\mathcal H$ be a unital $\FF$-algebra with a basis $\{[\lambda]\,:\,\lambda\in\Lambda\}$ such that 
$$
[\lambda]\cdot [\mu]\in \mathbb F^\times[\lambda+\mu]+\sum_{\nu\lhd\lambda+\mu} \mathbb F[\nu],
$$
(in particular, $[0]=1_{\mathcal H}$).
Then for any subset $\Lambda_{\circ}$ of~$\Lambda_{\min}$ which generates~$\Lambda$ as a monoid,
the set $[\Lambda_{\circ}]:=\{ [\lambda]\,:\,\lambda\in\Lambda_{\circ}\}$ generates~$\mathcal H$ as an algebra.
Moreover, for any total order on~$\Lambda_{\circ}$, the set $\mathbf M([\Lambda_{\circ}])$ of ordered monomials in~$[\Lambda_{\circ}]$
spans $\mathcal H$ as an $\FF$-vector space. Finally, if~$\Lambda$ is freely generated by~$\Lambda_{\circ}$ 
then $\mathbf M([\Lambda_{\circ}])$ is a basis of~$\mathcal H$.
\end{corollary}

\begin{proof}
Clearly, $\mathcal H$ is $\Lambda$-filtered with $\mathcal H^{\unlhd \lambda}=\FF\{[\mu]\,:\,\mu\unlhd \lambda\}$.
In particular, $\mathcal H^{\unlhd 0}=\FF[0]=\FF1_{\mathcal H}$.
Then $\gr_\Lambda\mathcal H$ has a basis $\{\overline{[\lambda]}\,:\,\lambda\in\Lambda\}$ and 
\begin{equation}\label{eq:mult}
\overline{[\lambda]}\cdot\overline{[\mu]}\in\mathbb F^\times\overline{[\lambda+\mu]},
\end{equation}
hence $\overline{[\lambda+\mu]}\in \mathbb F^\times \overline{[\lambda]}\cdot\overline{[\mu]}$. 

Let~$\le$ be any total order on~$\Lambda_{\circ}$. 
Given~$\lambda\in\Lambda$, we can write 
$\lambda=\lambda_1+\cdots+\lambda_r$ with $\lambda_i\in \Lambda_\circ$, $1\le i\le r$ and 
$\lambda_1\le \cdots\le \lambda_r$. By~\eqref{eq:mult} we have
$\overline{[\lambda]}\in\mathbb F^\times \overline{[\lambda_1]}\cdots \overline{[\lambda_r]}$. 
Taking into account that $\mathcal H^{\unlhd \lambda}=\FF[\lambda]$ for $\lambda\in\Lambda_{\circ}\cup\{0\}$, we see that 
all assumptions of Proposition~\ref{prop:filtered-graded} are satisfied. 
\end{proof}

\subsection{Proof of Proposition~\ref{prop:Iso-ordered}}\label{subs:part-ord-iso}
Let~$\cat A$ be a finitary exact category and let~$\lhd$ be the relation on~$(\Iso\cat A)^+$
defined by (cf.~\S\ref{subs:ord-mon-PBW-prop})
$$
[M]\lhd [M^-\oplus M^+]\iff \text{there exists a non-split short exact sequence $\xymatrix@C=2ex{M^-\ar@{>->}[r] &M\ar@{>>}[r] & M^+} $}.
$$
We extend it to~$\Iso\cat A$ by requiring that~$[0]\lhd [M]$ for all~$[M]\in(\Iso\cat A)^+$.
To prove that the transitive closure of~$\unlhd$ is a partial order, we use the following 
obvious fact.
\begin{lemma}\label{lem:order-by-func}
Let~$X$ be a set and $\mathcal R_0\subset (X\times X)\setminus D(X)$ where~$D(X)=\{(x,x)\,:\,x\in X\}$. 
Let~$\mathcal R$ be the transitive closure of~$\mathcal R_0\cup D(X)$ (hence a preorder).
Suppose that for any sequence $\mathbf x=(x_0,\dots,x_n)\subset X^n$ such that~$(x_{i-1},x_i)\in\mathcal R_0$, $1\le i\le n$,
 there exist a partially ordered set $(\mathcal P_{\mathbf x},\prec_{\mathbf x} )$ and 
 a function $f=f_{\mathbf x}:\{x_0,\dots,x_n\}\to\mathcal P_{\mathbf x}$, such that $f(x_0)\prec_{\mathbf x} f(x_1)\prec_{\mathbf x} \cdots
 \prec_{\mathbf x} f(x_n)$. 
 Then~$\mathcal R$ is a partial order on~$X$.
\end{lemma}
We apply Lemma~\ref{lem:order-by-func} to~$X=\Iso\cat A$ with $\mathcal R_0$ being the relation~$\lhd$.
Consider a sequence $\mathbf x=([M_0],\dots,[M_n])$ with $[M_{i-1}]\lhd[M_i]$, $1\le i\le n$. By definition of~$\lhd$, 
$[M_i]=[M_{i-1}^-\oplus M_{i-1}^+]$
where $\xymatrix@C=2ex{M^-_{i-1}\ar@{>->}[r] &M_{i-1}\ar@{>>}[r] & M^+_{i-1}}$ is a non-split short exact sequence.
Let $M_{\mathbf x}=M_0^+\oplus\cdots\oplus M_n^+$ and let 
$\mathcal P_{\mathbf x}$ be the Grothendieck monoid $\Gamma_{\mathbb Z-\fin }$ of the category $\mathbb Z-\fin$ of 
finite abelian groups. As follows e.g. from Proposition~\ref{prop:profinitary-ordered}, the natural order~$\prec$ on~$\mathbb Z-\fin$
(see~\S\ref{subs:order-gr-mon})
is a partial order. Set 
$$
f_{\mathbf x}([M_i])=|\Ext^1_{\cat A}(M_{\mathbf x},M_i)|,\qquad 0\le i\le n, 
$$
where~$|A|$ denotes the  image of~$A\in\mathbb Z-\fin$ in~$\Gamma_{\mathbb Z-\fin}$. We claim that~$f_{\mathbf x}([M_{i-1}])\prec
f_{\mathbf x}([M_i])$, $1\le i\le n$. 
To prove this claim, we need the following Lemma (cf.~\cite{GP}*{Lemma~2.1})

\begin{lemma}\label{lem:ineq-Z-fin}
Let~$\cat A$ be a finitary exact category. Then for any short exact sequence 
$$
\xymatrix{M^-\ar@{>->}[r]^{f_-} &M\ar@{>>}[r]^{f_+} & M^+}
$$
in~$\cat A$ and any object~$N\in\cat A$ the inequality
\begin{equation}\label{eq:ineq-Z-fin}
|\Ext^1_{\cat A}(N,M)|\preceq |\Ext^1_{\cat A}(N,M^+\oplus M^-)|
\end{equation}
holds in the partially ordered set~$\Gamma_{\mathbb Z-\fin}$. Moreover, if~\eqref{eq:ineq-Z-fin} is an equality
then the natural morphism of abelian groups $\Hom_{\cat A}(N,M)\to \Hom_{\cat A}(N,M^+)$, $h\mapsto 
f^+\circ h$, is surjective.
\end{lemma}

\begin{proof}

Since an exact category embeds into an abelian category as a full subcategory closed under extensions (see e.g.~\cite{Buh}),
given a morphism~$f\in\Hom_{\cat A}(A,B)$ and an object~$X\in\cat A$, we have a natural morphism 
of abelian groups $f_*:\Ext^1_{\cat A}(X,A)\to \Ext^1_{\cat A}(X,B)$. 
Moreover, 
the short exact sequence
$$
\xymatrix{M^-\ar@{>->}[r]^{f_-} &M\ar@{>>}[r]^{f_+} & M^+}
$$
gives rise to a complex of finite abelian groups 
\begin{equation}\label{eq:complex-ext}
0\to E_0\xrightarrow{d_0} E_1\xrightarrow{d_1} E_2\xrightarrow{d_2} E_3\xrightarrow{d_3} 0
\end{equation}
where~$E_0=\Im\delta$, $E_1=\Ext^1_{\cat A}(N,M^-)$, $E_2= \Ext^1_{\cat A}(N,M)$, $E_3=\Ext^1_{\cat A}(N,M^+)$,
$d_0$ is the natural embedding of~$\Im\delta$ into~$E_1$,
$d_1=(f_-)_*$ and $d_2=(f_+)_*$. Here~$\delta:\Hom_{\cat A}(N,M^+)\to\Ext^1_{\cat A}(N,M^-)$ is a canonical 
morphism of abelian groups.
The complex~\eqref{eq:complex-ext} is exact everywhere except at~$E_3$, that is
$$
\ker d_0=0,\quad \ker d_1=\Im d_0,\quad \ker d_2=\Im d_1,
$$
that is, its cohmologies~$H^i$, $0\le i\le 2$ satisfy
\begin{equation}\label{eq:cohom}
H^0=H^1=H^2=0,\qquad H^3=E_3/\Im d_2.
\end{equation}
Then, computing the Euler characteristic of~\eqref{eq:complex-ext}  in the Grothedieck group of~$\mathbb Z-\fin$ in two ways 
and applying~\eqref{eq:cohom}, we obtain 
$$
|E_0|-|E_1|+|E_2|-|E_3|=|H^0|-|H^1|+|H^2|-|H^3|=-|H^3|
$$
hence
\begin{equation}\label{eq:cohom-eq}
|E_2|+|\Im\delta|+|H^3|=|E_1|+|E_3|.
\end{equation}
and the following inequality holds in~$\Gamma_{\mathbb Z-\fin}$, which embeds into~$K_0(\mathbb Z-\fin)$
$$
|E_2|\preceq |E_1|+|E_3|,
$$
which is~\eqref{eq:ineq-Z-fin}. If~$|E_2|=|E_1|+|E_3|$ then in particular we must have~$\delta=0$ since~$\Gamma_{\mathbb Z-\fin}$
has the cancellation property. The assertion now follows since the image of the 
natural morphism $\Hom_{\cat A}(N,M)\to\Hom_{\cat A}(N,M^+)$
is the kernel of~$\delta$.
\end{proof}

Lemma~\ref{lem:ineq-Z-fin} immediately implies that~$f_{\mathbf x}([M_{i-1}])\preceq  f_{\mathbf x}([M_i])$. Moreover,
if we have an equality for some~$i$ then for every morphism $g\in\Hom_{\cat A}(M_{\mathbf x},M_i^+)$ there exists 
$h\in\Hom_{\cat A}(M_{\mathbf x},M_i)$ such that~$(f_+)_i\circ h=g$. For~$g=(0,\dots,0,1_{M_i^+},0,\dots,0)$ 
this implies that the short exact sequence 
$$
\xymatrix{M_i^-\ar@{>->}[r]^{(f_-)_i} &M\ar@{>>}[r]^{(f_+)_i} & M^+_i}
$$
splits, which is a contradiction. This proves the claim.

Thus, $f_{\mathbf x}$ satisfies the assumption of Lemma~\ref{lem:order-by-func} and so
the transitive closure of~$\unlhd$ is a partial order.

It remains to prove that $[M]\lhd[N]$, $[M],[N]\in(\Iso\cat A)^+$ implies that $[M\oplus X]\lhd[N\oplus X]$ for all 
$[X]\in\Iso\cat A$. If~$X=0$ there is nothing to do.
Otherwise, it is enough to observe that if a short exact sequence 
$\xymatrix{M^-\ar@{>->}[r]^{f_-} &M\ar@{>>}[r]^{f_+} & M^+}$ is non-split, then so is 
\begin{equation*}
\xymatrix@C=8ex{M^-\oplus X\ar@{>->}[r]^{\left(\begin{smallmatrix}f_-&0\\0&1_X\end{smallmatrix}\right)} &M\oplus X\ar@{>>}[r]^{(f_+\,\,0)} & M^+}.
\end{equation*}
This completes the proof of Proposition~\ref{prop:Iso-ordered}.\qed

\subsection{Proof of Theorem~\ref{thm:PBW-prop-Hall}}\label{pf:PBW}
We are now going to apply the machinery developed in~\S\ref{subs:filt-gr-ring}. We begin
by proving that~$(\Iso\cat A,\lhd)$ is optimal.
\begin{lemma}\label{lem:fin-indec}
Let~$\cat A$ be an exact $\Hom$-finite category. Then every object~$X$ in~$\cat 
A$ is a finite direct sum of indecomposable objects and
the number of indecomposable summands of~$X$ is bounded above by~$|\End_{\cat 
A}X|$.
\end{lemma}
\begin{proof}
Let~$X$ be a non-zero object in~$\cat A$. Write $X=X_1\oplus\cdots\oplus X_s$ 
for some~$s>0$, where all the~$X_i$ are non-zero. Then
$|\End_{\cat A}X|\ge \sum_{j=1}^s |\End_{\cat A}X_i|\ge s$. Let~$k$ be the 
maximal positive integer~$s$
such that $X$ can be written as a direct sum of~$s$ non-zero objects. The 
maximality of~$k$ immediately implies that each summand
is indecomposable.
\end{proof}
\begin{remark}
It should be noted that the Krull-Schmidt theorem does not have to hold in this 
generality. For example, the full subcategory 
of the category of $\kk$-representations of the quiver $1\to 0\leftarrow 2$ 
with the dimension vector satisfying $\dim_\kk V_1=\dim_\kk V_2$,
is not Krull-Schmidt.
\end{remark}

\begin{corollary}\label{cor:Iso-optimal}
The monoid~$\Iso\cat A$ is generated by~$\Ind\cat A$ and is optimal with respect to~$\unlhd$. 
\end{corollary}
\begin{proof}
The first assertion is immediate from the Lemma. To prove the second, observe that if~$[N]$ is not minimal,
then~$[M]\lhd[N]$ for some~$[M]\in\Iso\cat A$ and so $N$ is decomposable. 
Thus, every~$[X]\in\Ind\cat A$ is minimal with respect to the partial order~$\unlhd$,
hence~$\Iso\cat A$ is generated by its minimal elements.
\end{proof}
\begin{proof}[Proof of Theorem~\ref{thm:PBW-prop-Hall}]
Since~$(\Iso\cat A,\lhd)$ is optimal, the Hall algebra~$H_{\cat A}$ satisfies the assumptions of Corollary~\ref{cor:pre-PBW}
with~$\Lambda=\Iso\cat A$ and $\Lambda_\circ=\Ind\cat A$. Therefore, for any total order on~$\Ind\cat A$, 
ordered monomials on~$\Ind\cat A$ span~$H_{\cat A}$.
Finally, if~$\cat A$ is Krull-Schmidt, $\Iso\cat A$ is freely generated by~$\Ind\cat A$ hence ordered monomials on~$\Ind\cat A$
form a basis of~$H_{\cat A}$.
\end{proof}

\section{Grothendieck monoid of a profinitary category}\label{sec:main thm}

\subsection{Almost simple objects}
We will repeatedly need the following obvious description of the defining 
relation of the Grothen\-dieck monoid.
\begin{lemma}\label{lem:char-gr-mon}
Suppose that $[X]\not=[Y]\in(\Iso\cat A)^+$ and $|X|=|Y|$. Then 
there exist $[X_i]\in(\Iso\cat A)^+$, $0\le i\le r$ and 
$[A_i], [B_i]\in(\Iso\cat A)^+$, $1\le i\le r$ such that  
$[X_0]=[X]$, $[X_r]=[Y]$ and $[X_{i-1}],[X_{i}]\in\uExt^1_{\cat A}(A_i,B_i)$, 
$1\le i\le r$.
\end{lemma}

\begin{definition}\label{def:almost-simple}
We say that an object~$X\not=0$ in an exact category~$\cat A$ is {\em almost 
simple} if there is no 
non-trivial short exact sequence $Y\rightarrowtail X\twoheadrightarrow Z$ (or, 
equivalently, $[X]\in\uExt^1_{\cat A}(A,B)\implies
\{[A],[B]\}=\{[X],[0]\}$) 
and {\em simple} if it has no proper non-zero subobjects.
\end{definition}
Clearly, in an abelian category these notions coincide. Note that an almost simple 
object is always indecomposable. Let~$S_{\cat A}\subset
\Iso\cat A$ be the set of isomorphism classes of almost simple objects. The 
definition~\eqref{eq:coproductDelta} of comultiplication~$\Delta$
implies that 
$$
F^{AB}_{X}=\begin{cases}
            1,& \{[A],[B]\}=\{[X],[0]\}\\
0,&\text{otherwise},
           \end{cases}
$$
hence
\begin{equation}\label{eq:almost-simpl-prim}
S_{\cat A}\subset \Prim(H_{\cat A}).
\end{equation}

Let~$\Gamma$ be an abelian monoid.
Observe that the elements
of $\Gamma^+\setminus(\Gamma^+ + \Gamma^+)$ are precisely the minimal elements 
of~$\Gamma^+$ in the preorder~$\preceq$ (cf.~\S\ref{subs:order-gr-mon}).

\begin{lemma}\label{lem:min-indec}
Let~$\cat A$ be an exact category. Then the restriction of the canonical 
homomorphism of monoids $\phi_{\cat A}:\Iso\cat A\to \Gamma_{\cat A}$
to~$S_{\cat A}$ is a bijection 
\begin{equation}\label{eq:min-indec}
S_{\cat A}\to
\Gamma_{\cat A}^+\setminus(\Gamma_{\cat A}^++\Gamma_{\cat A}^+).
\end{equation}
In particular, $(H_{\cat A})_\gamma=\Prim(H_{\cat A})_{\gamma}$ for all 
$\gamma\in \Gamma_{\cat A}^+\setminus(\Gamma_{\cat A}^++\Gamma_{\cat A}^+)$
and is one-dimensional.
\end{lemma}
\begin{proof}
Lemma~\ref{lem:char-gr-mon} implies that for $[X]\in S_{\cat A}$, 
$|X|=|Y|$ if and only if~$[X]=[Y]$. This shows that
the restriction of $\phi_{\cat A}$ to~$S_{\cat A}$ is injective.
Furthermore, if~$|X|=|Y|+|Z|=|Y\oplus Z|$
for some non-zero $[Y],[Z]$ then $[X]=[Y\oplus Z]$ which is a contradiction 
since~$X$ is indecomposable.
Thus, $\Im\phi_{\cat A}\subset \Gamma_{\cat A}^+\setminus(\Gamma_{\cat 
A}^++\Gamma_{\cat A}^+)$ and so
\eqref{eq:min-indec} is well-defined.
Finally, 
if $[X]\notin S_{\cat A}$, then $[X]\in\uExt^1_{\cat A}(A,B)$ 
with~$|A|,|B|\in\Gamma_{\cat A}^+$, hence $|X|=|A|+|B|\in
\Gamma_{\cat A}^++\Gamma_{\cat A}^+$.
Thus, the preimage of
$\Gamma_{\cat A}^+\setminus(\Gamma_{\cat A}^++\Gamma_{\cat A}^+)$ is contained 
in~$S_{\cat A}$ hence $\phi_{\cat A}|_{S_{\cat A}}$ is surjective.

In particular, $\dim_{\QQ} (H_{\cat A})_\gamma=\#\Iso\cat A_\gamma=1$ for all 
$\gamma\in \Gamma_{\cat A}^+\setminus(\Gamma_{\cat A}^++\Gamma_{\cat A}^+)$.
The equality $(H_{\cat A})_\gamma=\Prim(H_{\cat A})_{\gamma}$ now follows 
from~\eqref{eq:almost-simpl-prim}.
\end{proof}
\begin{remark}
Note that we can have $\dim_\QQ(H_{\cat A})_\gamma=1$ even for 
$\gamma\in\Gamma^+ +\Gamma^+$. For example, if $S$, $S'$
are simple objects with $\Ext^1_{\cat A}(S,S')=\Ext^1_{\cat A}(S',S)=0$ then 
$(H_{\cat A})_{|S|+|S'|}=\QQ[S\oplus S']=
\QQ[S][S']$. However, in that case $\Prim(H_{\cat A})_\gamma=0$.
\end{remark}

\subsection{Proof of 
Proposition~\ref{prop:profinitary-ordered}}\label{pf:prop-profinitary-ordered}
Let
$$
\Gamma^f_{\cat A}=\{\gamma\in\Gamma_{\cat A}\,:\, \#\Iso\cat A_\gamma<\infty\}.
$$
Thus, $\cat A$ is profinitary if~$\Gamma_{\cat A}=\Gamma_{\cat A}^f$. Note, 
however, that~$\Gamma^f_{\cat A}$ need not be a submonoid of~$\Gamma_{\cat A}$.
Since~$\#\Iso\cat A_{\gamma}=1$ for~$\gamma\in\Gamma_{\cat A}^+$ minimal, all 
minimal elements 
of~$\Gamma_{\cat A}$ are contained in~$\Gamma_{\cat A}^f$.
The following Lemma is an immediate consequence of Lemma~\ref{lem:can-prop}
\begin{lemma}\label{lem:lower-set}
Let~$\gamma\in \Gamma^f_{\cat A}$. Then the set 
$$
\{\gamma'\in\Gamma_{\cat 
A}\,:\,\gamma'\preceq \gamma\}
$$
is contained in~$\Gamma^f_{\cat A}$.
\end{lemma}
\noindent
Given~$\gamma\in \Gamma_{\cat A}^f$, let 
$s_\gamma=\max_{[X]\in\Iso{\cat A}_\gamma} |\End_{\cat A}X|$.

Proposition~\ref{prop:profinitary-ordered} is a special case of the following
\begin{proposition}\label{prop:profinitary-ordered-1}
Let~$\cat A$ be a $\Hom$-finite exact category. Then
the restriction of the preorder~$\preceq$ to~$\Gamma_{\cat A}^f$ is a partial 
order. Moreover,
$\Gamma_{\cat A}^f$ is contained in the submonoid of~$\Gamma_{\cat A}$ 
generated by its minimal elements.
\end{proposition}
\begin{proof}
We need the following 
\begin{lemma}\label{lem:monoid-gen-simple}
Let~$\gamma\in\Gamma^f_{\cat A}\setminus\{0\}$. 
Then~$\gamma$ can be written as a sum of finitely many minimal elements of~$\Gamma_{\cat A}^+$
and the number of summands in any such presentation is bounded by~$s_\gamma$.
\end{lemma}
\begin{proof}
The proof is almost identical to that of Lemma~\ref{lem:fin-indec}.
Write $\gamma=\gamma_1+\cdots+\gamma_s$ for some~$\gamma_i\in\Gamma_{\cat A}^+$. 
Take~$X_i\in\cat A$ with $|X_i|=\gamma_i$ and let $X=X_1\oplus\cdots\oplus 
X_s$. 
Then~$s$ cannot exceed the maximal possible number of indecomposable summands 
of~$X$ which, by Lemma~\ref{lem:min-indec},
is bounded above by~$|\End_{\cat A}X|\le s_\gamma$. Let~$k$ be the maximal 
integer~$s$ such that 
$\gamma$ can be written as a sum of~$s$ elements of~$\Gamma^+_{\cat A}$. Then 
the maximality of~$k$ implies that each summand 
is minimal.
\end{proof}

It follows from Lemma~\ref{lem:monoid-gen-simple} that for~$\alpha\in\Gamma_{\cat A}^f$,
$\alpha=\alpha+\beta$ implies that~$\beta=0$. 
Then $\alpha+\beta+\gamma=\alpha$ implies that $\beta+\gamma=0$, hence $\beta=\gamma=0$ since~$0$
is the only invertible element of~$\Gamma_{\cat A}$. The first assertion of the Proposition
now follows from Lemma~\ref{lem:nat-ord},
while the second is immediate from Lemma~\ref{lem:monoid-gen-simple}.
\end{proof}

\subsection{Proofs of Theorems~\ref{th:prof-cof}, \ref{th:prof-ab} and Corollary~\ref{cor:prof-cof}}\label{subs:pf profinitary 
bifinitary}
We begin by proving Theorem~\ref{th:prof-cof}. 
\begin{proof}
Since~$[A\oplus B]\in\uExt^1_{\cat A}(A,B)$, 
Definition~\ref{def:prof-cof-bif} implies that a profinitary category~$\cat A$ is cofinitary if and only 
for any~$\gamma\in\Gamma_{\cat A}$ the set 
\begin{multline*}
\mathcal E_\gamma:=\{([A],[B])\in\Iso\cat A\times\Iso\cat A\,:\, |A|+|B|=\gamma\}\\=\bigcup_{[X]\in\Iso\cat A\,:\, |X|=\gamma}
\{([A],[B])\in\Iso\cat A\times\Iso\cat A\,:\, [X]\in \uExt_{\cat A}^1(A,B)\}
\end{multline*}
is finite. On the other hand, $$
\mathcal E_\gamma=\bigcup_{\alpha,\beta\in\Gamma_{\cat A}\,:\,\alpha+\beta=\gamma}\Iso\cat A_\alpha\times\Iso\cat A_\beta.
$$
Therefore, $\mathcal E_\gamma$ is finite if and only if~$\{(\alpha,\beta)\in\Gamma_{\cat A}\times\Gamma_{\cat A}\,:\,
\alpha+\beta=\gamma\}$ is finite.
\end{proof}

Now we proceed to prove Theorem~\ref{th:prof-ab}. 
Given an object~$X\in\cat A$, an admissible flag on $X$ is a sequence of 
objects $X_0=X$, $X_1,\dots,X_s=0$ together 
with short exact sequences $\xymatrix@C=3ex{X_{i}\ar@{>->}[r]& 
X_{i-1}\ar@{->>}[r]&Y_{i}}$, $1\le i\le s$. An admissible flag is said to be a 
{\em composition series}
if the $Y_i$ are almost simple for all~$1\le i\le s$.

\begin{proposition}\label{prop:comp-ser}
Let $\cat A$ be a $\Hom$-finite exact category.
Suppose that~$\gamma\in\Gamma_{\cat A}^f\setminus\{0\}$. 
Then $X\in\cat A$ with $|X|=\gamma$ admits 
a composition series. Moreover, the length of any composition series of~$X$ 
is bounded above by $s_\gamma$.
\end{proposition} 
\begin{proof}
We proceed by induction on the partially order set~$(\Gamma^f_{\cat A},\preceq)$ (see Proposition~\ref{prop:profinitary-ordered-1}). 
If~$\gamma\in\Gamma^f_{\cat A}$ 
is minimal then~$X$ with~$|X|=\gamma$ is almost simple by 
Lemma~\ref{eq:min-indec} hence admits a 
composition series.
Suppose the assertion is established for all~$\gamma'\prec\gamma$ and~$\gamma$ 
is not minimal. Then~$X$ with~$|X|=\gamma$ is not almost simple,
hence there exists a short exact sequence 
$\xymatrix@C=3ex{X''\ar@{>->}[r]& X\ar@{->>}[r]^h&X'}$ with $|X'|,|X''|\prec |X|$. By 
Lemma~\ref{lem:lower-set}, $|X'|\in\Gamma^f_{\cat A}$, hence by the induction hypothesis
there exists a short exact sequence $\xymatrix@C=3ex{Y''\ar@{>->}[r]& X'\ar@{->>}^g[r]&Y}$ with~$Y$ almost 
simple. Let~$Y_1=Y$. Then we have a short exact sequence 
$$
\xymatrix{X_1\ar@{>->}[r]&X\ar@{->>}[r]^{gh}&Y_1}
$$
where $|X_1|\prec |X|$ hence $|X_1|\in\Gamma^f_{\cat A}$. Therefore, $X_1$
admits a composition series by the induction hypothesis, which establishes 
the first assertion of the Lemma.
The second assertion is immediate from 
Lemma~\ref{lem:monoid-gen-simple} 
since $|X|=|Y_1|+\cdots+|Y_s|$.
\end{proof}

\begin{proof}[Proof of Theorem~\ref{th:prof-ab}]
If $\cat A$ is a profinitary and abelian, then the composition series from Proposition~\ref{prop:comp-ser} is
a composition series in the usual sense since all almost simple objects are simple. Theorem~\ref{th:prof-ab} is now immediate.
\end{proof}

Finally, we can prove Corollary~\ref{cor:prof-cof}.
\begin{proof}
Since a full exact subcategory of a cofinitary exact category 
is also cofinitary, to prove~\eqref{cor:prof-cof.a}, it suffices to consider the case when~$\cat A$ is a 
profinitary abelian category. 
Note that the uniqueness of composition factors in an abelian category with the finite length property 
 (see e.g.~\cite{Joy}*{Theorem~2.7}) implies that  
$\Gamma_{\cat A}$ is freely generated by its minimal elements. It remains to 
apply Theorem~\ref{th:prof-cof}. To prove~\eqref{cor:prof-cof.c},
note that by Lemma~\ref{lem:monoid-gen-simple}, $\Gamma_{\cat A}$ is finitely generated if and only if it contains 
finitely many minimal elements~$\gamma_1,\dots,\gamma_n$. 
Again by Lemma~\ref{lem:monoid-gen-simple}, the number 
of decompositions of~$\gamma\in\Gamma_{\cat A}$ as $\gamma=\sum_{i=1}^n c_i \gamma_i$, $c_i\in\mathbb Z_{\ge 0}$ is bounded above 
by $\binom{s_\gamma+n}{n}$, which is the number
of $n$-tuples $(c_1,\dots,c_n)\in\mathbb Z_{\ge 0}^n$ with $\sum_{i=1}^n c_i\le s_\gamma$. The assertion is now immediate from 
Theorem~\ref{th:prof-cof}.
\end{proof}

\section{Coalgebras in tensor categories and proof of main theorem}

\subsection{Quasi-primitive elements and co-ideals}
Let~$\FF$ be a field of characteristic zero and let~$\cat C$ be an $\FF$-linear tensor category.
For any coalgebra $C$ in $\cat C$ denote by $C_0=\operatorname{Corad}_{\cat 
C}(C)$ the sum of all simple finite-dimensional sub-coalgebras of~$C$
in $\cat C$ and refer to it as the {\em coradical} of $C$ in ${\cat C}$. 
Clearly, 
$C_0$ is a sub-coalgebra of $C$ in ${\cat C}$. Denote also 
$$C_1=\QPrim_{\cat C}(C)=\Delta^{-1}(C\otimes C_0+C_0\otimes C)$$ 
and refer to it as the  quasi-primitive set of $C$. Then $C_1$ is a $\cat C$-subobject 
of~$C$. It is well-known (\cite{Swe}*{Corollary~9.1.7}) that
$$
\Delta(C_1)\subset C_1\otimes C_0+C_0\otimes C_1.
$$
In particular, if $C_0=\FF$ then $\QPrim_{\cat C}(C)=\FF\oplus \Prim(C)$.  
More generally, we have the following Lemma which extends a well-known result 
(cf.~\cite{M}*{Theorem 5.2.2}, \cite{Swe}*{\S9.1}). 
\begin{lemma}  
\label{le:coradical} Any coalgebra $C$ in ${\cat C}$  admits an 
increasing coradical filtration by sub-coalgebras $C_k\subset C$ in ${\cat C}$, 
$k\ge 0$, defined by: 
$C_0=\operatorname{Corad}_{\cat C}(C)$, $C_1=\QPrim_{\cat C}(C)$ and 
$$C_k=\Delta^{-1}(C\otimes C_{k-1}+C_{0}\otimes C)$$
for $k>1$. Moreover, $\Delta(C_k)=\sum_{i=0}^k C_i\tensor C_{k-i}$.\qed
\end{lemma}

A coideal in~$C$ is a $\cat C$-subobject~$I$ of~$C$ satisfying $$
\Delta(I)\subset C\tensor I+I\tensor C.
$$
\begin{proposition}\label{pr:coid-intersects-qprim}
Let $C$ be a coalgebra in ${\cat C}$. Then for any non-zero coideal $I$ in 
${\cat C}$ 
one has 
$$
I\cap \QPrim_{\cat C}(C)\ne \{0\}.
$$
\end{proposition} 

\begin{proof} 
For each $k\ge 0$ denote $I_k:=I\cap C_k$. If~$I_0\not=\{0\}$ then we are done 
since~$I_0\subset C_0\subset C_1$.
Assume that~$I_0=0$. %Since~$I$ is a coideal, $\Delta(I)\subset C\tensor I+I\tensor C$.
Since~$C_0\subset C_1\subset\cdots$ is a filtration, there exists a unique $k\ge 1$ such that $I_{k-1}=0$ and~$I_k\not=0$. Then
$$\Delta(I_{k})\subset C_0 \otimes I_{k}+I_{k} \otimes C_0 \ .$$
Since $C_1$ is the maximal subobject~$V$ of~$C$ with the 
property~$\Delta(V)\subset C_0\tensor V+V\tensor C_0$,
it follows that $I_{k}\subset C_1$ and thus $k=1$. Therefore, $I_1=I\cap C_1=I\cap\QPrim_{\cat C}(C)\ne 
\{0\}$. 
\end{proof}
\subsection{Invariant pairing}
Let $H_0$ be a bialgebra over~$\FF$ and 
let~$\cat C$ be the category of left $H_0$-comodules. Given $V\in\cat C$, we denote 
the left co-action of~$H_0$ by $\delta_V:V\to H_0\tensor V$ and write, using Sweedler-like notation 
$$
\delta_V(v)=v^{(-1)}\tensor v^{(0)},\qquad v\in V.
$$
The category~$\cat C$ is an $\FF$-linear tensor category 
with 
the tensor product $A\otimes B=A\tensor_\FF B$ of objects $A,B\in {\cat C}$ 
acquiring a left $H_0$-comodule structure via
$$\delta_{A\tensor B}(a\otimes b)=a^{(-1)}b^{(-1)}\otimes a^{(0)}\otimes 
b^{(0)},
$$
for all $a\in A$, $b\in B$. 

Given two objects $A,B$ in ${\cat C}$, we  say that a paring $\langle 
\cdot,\cdot\rangle:A\otimes B\to \FF$ is $H_0$-{\em invariant} if 
$$a^{(-1)}\langle a^{(0)},b\rangle=b^{(-1)}\langle a ,b^{(0)}\rangle$$
for all $a\in A$, $b\in B$.

The following example plays a fundamental role in the sequel. 
\begin{example}\label{ex:monoidal coalg}
Let~$\Gamma$ be an 
abelian monoid. Its 
monoidal algebra 
$H_0=\FF\Gamma$ has a natural coalgebra structure, with the elements of~$\Gamma$ 
being group-like. 
Then a left $H_0$-comodule~$V$ is in fact a $\Gamma$-graded vector space, since 
$V=\bigoplus_{\gamma\in \Gamma} V_\gamma$
where $V_\gamma=\{ v\in V\,:\, \delta_V(v)=\gamma\tensor v\}$. It is easy to see 
that a pairing 
$\lr{\cdot,\cdot}:A\tensor B\to \FF$ is
$H_0$-invariant if and 
only if $\lr{A_\gamma,B_{\gamma'}}=0$, $\gamma\not=\gamma'\in\Gamma$.
\end{example}

\begin{lemma}\label{lem:perp-subobj} Let  $\langle \cdot,\cdot\rangle:A\otimes B\to \FF$ 
be a $H_0$-invariant pairing between objects $A$ and $B$ of ${\cat C}$. Then 
for any subobject $A_0$ of $A$ in~$\cat C$ its right orthogonal complement
$$A_0{}^\perp=\{b\in B\,:\,\langle A_0,b\rangle=0\}$$
is a subobject of $B$ in ${\cat C}$. Likewise, for any sub-object $B_0$ of~$B$ in~$\cat C$ its 
left orthogonal 
complement 
$$
{}^\perp B_0=\{a\in A\,:\, \lr{a,B_0}=0\}
$$
is a subobject of~$A$ in~$\cat C$.
\end{lemma}
\begin{proof}
We prove the first assertion only, the argument for the second one being 
similar. Given~$b\in A_0{}^\perp$, write $\delta_B(b)=\sum_{i} h_i\tensor b_i$
where the $h_i\in H_0$ are linearly independent and $b_i\in B$. Since the 
pairing is invariant, we have for all $a\in A_0$,
$$
\sum_i h_i \lr{a,b_i}=a^{(-1)}\lr{a^{(0)},b}=0
$$ 
since $\delta_A(a)=a^{(-1)}\tensor a^{(0)}\in H_0\tensor A_0$. Therefore, 
$\lr{a,b_i}=0$ for all~$i$, hence $b_i\in A_0{}^\perp$ and so $\delta_B(b)\in
H_0\tensor A_0{}^\perp$.
\end{proof}

We now prove that an $H_0$-invariant pairing between non-isomorphic simple objects 
in~$\cat C$ must be identically zero. For that purpose, it will be convenient to introduce the dual picture. Let~$H_0^*=\Hom_{\FF}(H_0,\FF)$. Then~$H_0^*$ 
is an associative $\FF$-algebra via $f\cdot g=(f\tensor g)\circ \Delta_{H_0}$ for all~$f,g\in H_0^*$, where $\Delta_{H_0}:H_0\to
H_0\tensor H_0$ is the 
comultiplication on~$H_0$ (hereafter we identify $\mathbb F\tensor_{\mathbb F} V$ with~$V$ via the canonical isomorphism).
Then a left~$H_0$-comodule $V$ is naturally a left $H_0^*$-module via 
$f\triangleright v=(f\tensor 1)\delta_V(v)$, for all $f\in H_0^*$ and~$v\in V$. This yields a fully faithful functor from the category $\cat C$
to the category of left $H_0^*$-modules. In particular, $V\cong V'$ in~$\cat C$ if and only if they 
are isomorphic as $H_0^*$-modules.
If $\lr{\cdot,\cdot}:A\tensor B\to \FF$ is an $H_0$-invariant pairing, then for all~$a\in A$, $b\in B$ and~$f\in H_0^*$ we have 
\begin{equation}\label{eq:form-alg}
\begin{split}
\lr{f\triangleright a,b}&=f(a^{(-1)})\lr{a^{(0)},b}=f(b^{(-1)})\lr{a,b^{(0)}}\\
&=\lr{a,f\triangleright b}.
\end{split}
\end{equation}
Finally, note that~$V$ is a simple $H_0$-comodule if and only if it is simple as a left $H_0^*$-module.
\begin{proposition}\label{prop:simple-orth}
Let $A$ and~$B$ be simple objects in~$\cat C$. Let~$\lr{\cdot,\cdot}:A\tensor B\to \FF$ be a non-zero $H_0$-invariant pairing. 
Then
 $A\cong B$ in~$\cat C$.
\end{proposition}
\begin{proof}
Given $a\in A$, let~$J_a=\Ann_{H_0^*} a=\{ f\in H_0^*\,:\,f\triangleright  a=0\}$.
We need the following technical result.
\begin{lemma}\label{lem:annihilators}
Let $A,B$ be objects in~$\cat C$ and let $\lr{\cdot,\cdot}:A\tensor B\to \FF$ be an $H_0$-invariant pairing 
such that ${}^\perp B=0$. If $B$ is simple, then~$J_a\subset \Ann_{H_0^*}B$ for all $a\in A$, $a\not=0$. Moreover,
if~$A$ is also simple, then $J_a=\Ann_{H_0^*}A$.
\end{lemma}
\begin{proof}
Let $a\in A$, $a\not=0$ and take~$f\in J_a$. It follows from~\eqref{eq:form-alg} that 
$0=\lr{f\triangleright a,b}=\lr{a,f\triangleright b}$ for all~$b\in B$. Thus,
$\lr{a,J_a\triangleright B}=0$, hence $a\in {}^\perp( J_a\triangleright B)$. 
Since~${}^\perp B=0$, this implies that $J_a\triangleright B$ is a proper $H_0^*$-submodule of~$B$, hence 
$J_a\triangleright B=0$ by the simplicity of~$B$.

Suppose now that~$A$ is also simple. Then $J_a$ is a maximal left ideal for all~$a\not=0$. 
If~$J_a\not=J_{a'}$ for some~$a,a'\in A$ then $J_a+J_{a'}=H_0^*\ni 1$, hence~$B=0$, which contradicts the simplicity of~$B$. Thus,
$J_a=J_{a'}$ for all $a,a'\in A$ and so $\Ann_{H_0^*}A=\bigcap_{a'\in A}J_{a'}=J_a$.
\end{proof}

Since $A$, $B$ are simple and the form~$\lr{\cdot,\cdot}$ is $H_0$-invariant and non-zero, ${}^\perp B=0$ by Lemma~\ref{lem:perp-subobj}.
Then $\Ann_{H_0^*}A\subset \Ann_{H_0^*}B$ by Lemma~\ref{lem:annihilators}. Let~$R=H_0^*/\Ann_{H_0^*}A$. Then both~$A$ and~$B$ 
are $R$-modules in a natural way and are simple as such.
Moreover, $A\cong B$ as $H_0$-comodules if and only if $A\cong B$ as $R$-modules. Furthermore, by definition of~$R$ and
Lemma~\ref{lem:annihilators} every non-zero element of~$R$
acts on~$A$ by an injective $\FF$-linear endomorphism. Since $A$ is a simple $H_0$-comodule, it is finite dimensional
(see e.g.~\cite{M}*{Corollary~5.1.2}). Thus,
each non-zero element of~$R$ acts on~$A$ by an $\FF$-automorphism. This implies that
$R$ is a division algebra, hence admits a unique, up to an isomorphism, simple finite dimensional module and so
$A\cong B$ as $R$-modules. Therefore, $A\cong B$ as objects in~$\cat C$. 
\end{proof}

Denote by $\cat C^f$ the full subcategory of~$\cat C$ whose objects are direct 
sums of simple comodules with finite multiplicities; thus,
an object $V\in\cat C^f$ can be written as 
$V=\bigoplus_{{\boldsymbol i}\in\boldsymbol I} V_{\boldsymbol i}$ where each 
$V_{\boldsymbol i}$ is a finite direct sum of isomorphic simple subcomodules 
of~$V$ 
hence by~\cite{M1} is finite-dimensional.

\begin{lemma}\label{lem:lem4.14}
 Suppose that $V=\bigoplus_{{\boldsymbol i}\in\boldsymbol I} V_{\boldsymbol i}\in\cat C^f$ admits a $H_0$-invariant bilinear 
form $\lr{\cdot,\cdot}:V\tensor V\to \FF$. 
Then for any subobject~$U$ of~$V$ in~$\cat C$,
$$
U^\perp\supset\bigoplus_{\boldsymbol i\in \boldsymbol I} U_{\boldsymbol i}^\perp,
$$
where $U_{\boldsymbol i}^\perp=\{ v\in V_{\boldsymbol i}\,:\, \lr{U\cap V_{\boldsymbol i}, v}=0\}$.
\end{lemma}
\begin{proof}
By Proposition~\ref{prop:simple-orth}, $\lr{V_{\boldsymbol i},V_{\boldsymbol j}}=0$ if $\boldsymbol i\not=\boldsymbol j$. The 
assertion is now immediate.
\end{proof}

\subsection{Quasi-primitive generators}
\begin{definition} 
\label{def:compatible pairing}
Let $(A,\cdot,1)$ be a unital algebra and $(B,\Delta,\varepsilon)$ be a coalgebra in 
${\cat C}$. 
We say that an $H_0$-invariant pairing  $\langle \cdot,\cdot\rangle:A\otimes B\to \FF$ is 
compatible with $(A,\cdot ,1)$ and $(B,\Delta,\varepsilon)$
if 
$$\langle a\cdot a', b\rangle=\langle a\otimes a',\Delta(b)\rangle,\qquad 
\varepsilon(b)=\lr{1,b}
$$
for all $a,a'\in A$, $b\in B$, where the pairing $\lr{\cdot,\cdot}:(A\otimes A)\otimes (B\otimes 
B)\to \FF$ is defined by 
$$\langle a\otimes a',b\otimes b'\rangle=\langle a, b'\rangle \langle 
a',b\rangle.$$   
\end{definition}

The main ingredient in our proof of Theorem~\ref{th:prim generation} is the following result.
\begin{theorem} 
\label{th:Nichols characterization}
Let $A$ be an algebra (denoted by $(A,\cdot,1)$) and a coalgebra (denoted by 
$(A,\Delta,\varepsilon)$) in ${\cat C}^f$. Let
$
\langle\cdot,\cdot\rangle:A\otimes A\to \FF
$ 
be a compatible pairing between $(A,\cdot,1)$ and $(A,\Delta,\varepsilon)$ satisfying
$\lr{a,a}\not=0$ for all $a\in A\setminus\{0\}$.
Then $(A,\cdot,1)$ is generated by $A_1=\QPrim(A,\Delta,\varepsilon)$. 
\end{theorem}

\begin{proof} 
Let~$B$ be a $\cat C$-subalgebra of~$A$.
Since~$A\in\cat C^f$ and~$B$ is its subobject, $A=\bigoplus_{\boldsymbol i} A_{\boldsymbol i}$ and
$B=\bigoplus_{\boldsymbol i} B_{\boldsymbol i}$ where $B_{\boldsymbol i}=A_{\boldsymbol i}\cap B$.
By Proposition~\ref{prop:simple-orth}, $\lr{A_{\boldsymbol i},A_{\boldsymbol j}}=0$ for all $\boldsymbol i\not=\boldsymbol j$.
We claim that $B^\perp$ is a coideal of~$A$ in~$\cat C$.

Indeed, %let~$x\in B^\perp$. Then for any~$b_{\boldsymbol i}\in B_{\boldsymbol i}$,
for any $\boldsymbol i$, $\boldsymbol j$ 
we have 
$$
\{0\}=\lr{B_{\boldsymbol i}\cdot B_{\boldsymbol j},B^\perp}=\lr{B_{\boldsymbol i}\tensor 
B_{\boldsymbol j},\Delta(B^\perp)}.
$$
Thus, 
$\Delta(B^\perp)\subset\bigoplus_{\boldsymbol i,\boldsymbol j} (B_{\boldsymbol i}\tensor B_{\boldsymbol j})^\perp$ 
where $(B_{\boldsymbol i}\tensor B_{\boldsymbol j})^\perp=\{z\in A_{\boldsymbol 
j}\tensor A_{\boldsymbol i}\,:\, \lr{ B_{\boldsymbol i}\tensor B_{\boldsymbol j},z}=0\}$.
We need the following simple fact from linear algebra
 \begin{lemma}\label{lem:lin-alg} Let $U$, $V$ be finite dimensional vector spaces over~$\FF$ and $U'\subset U$, $V'\subset V$
be their subspaces. Then
\begin{enumerate}[{\rm(a)}]
 \item\label{lem:lin-alg.a} $U'\tensor V'=(U'\tensor V)\cap (U\tensor V')$;
\item\label{lem:lin-alg.b} For any subspaces $V_1,V_2$ of~$V$ 
$$
(V_1\cap V_2)^\perp=V_1^\perp+V_2^\perp,\qquad V_1^\perp\cap V_2^\perp=(V_1+V_2)^\perp,
$$
where~$W^\perp=\{ f\in V^*\,:\, f(W)=0\}$ for any subspace~$W\subset V$;
\item\label{lem:lin-alg.c} $
(U'\tensor V')^\perp=V'{}^\perp \tensor U^*+V^*\tensor U'{}^\perp$,
where we canonically identify $(U\tensor V)^*$ with~$V^*\tensor U^*$.
\end{enumerate}
\end{lemma}
\begin{proof}
Parts~\eqref{lem:lin-alg.a} and~\eqref{lem:lin-alg.b} are easily checked. To prove~\eqref{lem:lin-alg.c},
note that $(U'\tensor V)^\perp=V^*\tensor U'{}^\perp$ and~$(U\tensor V')^\perp=V'{}^\perp\tensor U^*$, hence by parts~\eqref{lem:lin-alg.a}
and~\eqref{lem:lin-alg.b}
\begin{equation*}
(U'\tensor V')^\perp=((U'\tensor V)\cap(U\tensor V'))^\perp=(U'\tensor V)^\perp+(U\tensor V')^\perp=V^*\tensor U'{}^\perp+V'{}^\perp\tensor U^*.
\qedhere
\end{equation*}
\end{proof}
Since $A_{\boldsymbol k}$ is finite dimensional and the restriction of~$\lr{\cdot,\cdot}$ to $A_{\boldsymbol k}$ 
is non-degenerate,
we naturally identify $A_{\boldsymbol k}^*$ with~$A_{\boldsymbol k}$
via $a\mapsto f_a$, where $f_a(a')=\lr{a',a}$. Then, applying Lemma~\ref{lem:lin-alg}\eqref{lem:lin-alg.c} with 
$U=A_{\boldsymbol i}$, $V=A_{\boldsymbol j}$, $U'=B_{\boldsymbol i}$ and~$V'=B_{\boldsymbol j}$,
we obtain
\begin{equation}\label{eq:pij}
(B_{\boldsymbol i}\tensor B_{\boldsymbol j})^\perp=
A_{\boldsymbol j}\tensor B^\perp_{{\boldsymbol i}}{}+B^\perp_{{\boldsymbol 
j}}\tensor A_{\boldsymbol i},
\end{equation}
where $B_{\boldsymbol k}^\perp=\{ a\in A_{\boldsymbol k}\,:\, \lr{B_{\boldsymbol k},a}=0\}$.
We conclude that $\Delta(B^\perp)\subset\bigoplus_{\boldsymbol i,\boldsymbol j} (B_{\boldsymbol i}\tensor B_{\boldsymbol j})^\perp\subset
A\tensor B^\perp+B^\perp\tensor A$. 
To complete the proof of the claim, observe that $\varepsilon(B^\perp)=\lr{1,B^\perp}=0$.

Now we complete the proof of Theorem~\ref{th:Nichols 
characterization}.
Let  $B$ be the subalgebra of $A$ in ${\cat C}$ generated by the subobject 
$A_1=\QPrim_{\cat C}(A)$ of~$A$ and suppose that $B\not=A$. Then, by the above claim, the orthogonal complement $I=B^\perp $ 
is a coideal of $A$ in $\cat C$. By Lemma~\ref{lem:lem4.14}, $I\supset \bigoplus_{\boldsymbol i} B_{\boldsymbol i}^\perp\not=\{0\}$
because $B_{\boldsymbol i}\not=A_{\boldsymbol i}$ for some~$\boldsymbol i$.
Therefore, $I\not=\{0\}$ hence $I\cap A_1\not=\{0\}$ by Proposition~\ref{pr:coid-intersects-qprim}. 
Yet $I\cap B=\{0\}$ since
$\lr{x,x}\not=0$ for all $x\in A$, hence
$I\cap A_1=\{0\}$ and we obtain a contradiction. Thus, $B=A$. 
\end{proof}

\subsection{Proof of Theorems~\ref{th:profinitary 
primitives}/\ref{th:prim generation}}\label{pf:main theorem}
Let 
$\FF=\QQ$ and define Green's pairing 
$\langle\cdot,\cdot\rangle:H_{\cat A}\tensor H_{\cat A}\to\QQ$ (cf.~\cite{Green}) by
\begin{equation}
\label{eq:hopf pairing}
\langle[A],[B]\rangle =\frac{\delta_{[A],[B]}}{|\Aut_{\cat A}(A)|}
\end{equation}
for any $[A],[B]\in {\Iso\cat A}$.

Clearly, this pairing is positive definite and symmetric. We extend 
$\langle\cdot,\cdot\rangle$ to a symmetric bilinear form on $H_{\cat A}\otimes H_{\cat 
A}$ by
$$\langle [A]\otimes [B],[C]\otimes 
[D]\rangle=\langle[A],[D]\rangle\langle[B],[C]\rangle$$
for any $[A],[B],[C],[D]\in {\Iso\cat A}$. 

\begin{lemma} 
Let~$\cat A$ be a cofinitary category.
Then \eqref{eq:hopf pairing} is a compatible pairing (in the sense 
of Definition \ref{def:compatible pairing}) between 
the Hall algebra $H_{\cat A}$ and the coalgebra $(H_{\cat 
A},\Delta,\varepsilon)$. 
\end{lemma}

\begin{proof} We abbreviate $\Gamma=\Gamma_{\cat A}$ and let ${\cat C}={\cat 
C}_\Gamma$ be the category of $\Gamma$-graded vector spaces or, equivalently, $\QQ\Gamma$-comodules (cf.~Example~\ref{ex:monoidal coalg}). 
It follows immediately 
from Example~\ref{ex:monoidal coalg} that the 
pairing \eqref{eq:hopf pairing} is $\QQ{\Gamma}$-invariant.

It remains to prove the compatibility in the sense of Definition~\ref{def:compatible pairing}, that is 
$$\langle [A]\cdot [B],[C]\rangle=\langle [A]\otimes [B],\Delta([C])\rangle$$
for all $[A],[B],[C]\in {\Iso\cat A}$. 
Indeed, 
\begin{align*}
\langle[A]\cdot [B],[C]\rangle=\frac{F_{A,B}^C}{|\Aut_{\cat A}(C)|}&= 
\frac{F^{B,A}_C}{|\Aut_{\cat A}(B)|\, |\Aut_{\cat A}(A)|}=\sum_{[A'],[B']}  
F^{B',A'}_C \langle [A],[A']\rangle  \langle [B],[B']\rangle\\
&=\sum_{[B'],[A']}  F^{B',A'}_C \langle [A]\tensor [B],[B']\otimes [A']\rangle
=\langle [A]\otimes [B],\Delta([C])\rangle.\qedhere
\end{align*}

\end{proof}

\begin{proof}[Proof of Theorems~\ref{th:profinitary primitives} and~\ref{th:prim generation}]
Suppose that $\cat A$ is profinitary and cofinitary. Since for each 
$\gamma\in\Gamma=\Gamma_{\cat A}$, $(H_{\cat A})_\gamma$ is finite dimensional and hence
is a finite direct sum of isomorphic simple left $\mathbb Q\Gamma$-comodules, 
$H_{\cat A}\in \cat 
C^f_{\Gamma}$. Then, clearly, $A=H_{\cat A}$
and the pairing~\eqref{eq:hopf pairing} satisfy 
all the assumptions of Theorem~\ref{th:Nichols characterization}. Therefore, 
$H_{\cat A}$ is generated by $A_1=\QPrim(H_{\cat A},\Delta,\varepsilon)$ in 
${\cat C}_{\Gamma}$.

Our next step is to show that~$A_1=\Prim(H_{\cat A},\Delta,\varepsilon)$, which 
gives the first assertion of Theorem~\ref{th:profinitary primitives}. For that, we need 
the following result.
\begin{lemma}\label{lem:subcoalg} Let $C=\bigoplus_{\gamma\in \Gamma} C_{\gamma}$ be a coalgebra in 
the category ${\cat C}_\Gamma$. Assume that for every~$\gamma\in\Gamma^+$, there exists~$h_\gamma\in\mathbb Z_{> 0}$ 
such that $\gamma$ cannot be written as a sum of more than~$h_\gamma$ elements of~$\Gamma^+$.
Then $\operatorname{Corad}_{\cat C}(C)\subset C_{\bf 0}$ where~$\bf0$ is the zero element of~$\Gamma$.
\end{lemma}
\begin{proof}
First, observe that~$\mathbf0$ is the only invertible element of~$\Gamma$, since otherwise $\mathbf0=\alpha+\beta$ for some~$\alpha,\beta\in\Gamma^+$ 
and so $\alpha=(n+1)\alpha+n\beta$ for any~$n\in\mathbb Z_{>0}$, which is a contradiction.
Since for any subcoalgebra $D=\bigoplus_{\gamma\in\Gamma} D_\gamma$ of~$C$ in~$\cat 
C_\Gamma$
$$
\Delta(D_\gamma)\subset \bigoplus_{\gamma',\gamma''\in \Gamma\,:\, 
\gamma=\gamma'+\gamma''} D_{\gamma'}\tensor D_{\gamma''}.
$$
it follows that $\Delta(D_{\bf 0})\subset D_{\bf 0}\tensor D_{\bf 0}$. 
Therefore, $D_{\bf 0}$ is a subcoalgebra of~$D$. 

We claim that~$D=0$ if and only if~$D_{\mathbf 0}=0$.
Indeed, if~$D_{\bf 0}=0$ then, since for the $k$th iterated 
comultiplication~$\Delta^k$ we have
$$
\Delta^{k}(D_\gamma)\subset\sum_{\gamma_0,\dots,\gamma_k\in\Gamma\,:\,\gamma_0+\cdots+\gamma_k=\gamma} D_{\gamma_0}\tensor\cdots\tensor D_{\gamma_k},
$$
it follows that $\Delta^{h_\gamma}(D_\gamma)=0$, since then in each summand we must have~$\gamma_i=0$ for some~$0\le i\le h_\gamma$ by the assumptions of the Lemma.
This implies that~$D_\gamma=0$ for all~$\gamma\in\Gamma$, hence~$D=0$. The converse is obvious.

Thus, if~$D$ is a simple subcoalgebra of~$C$, then $D_{\bf 0}\not=0$ and so~$D=D_{\bf 0}$.
\end{proof}

By Lemma~\ref{lem:monoid-gen-simple}, $\Gamma_{\cat A}$
satisfies the assumptions of Lemma~\ref{lem:subcoalg}, with $h_\gamma\le s_\gamma$,
hence $\operatorname{Corad}_{\cat C}(H_{\cat 
A})=\QQ$ and $\QPrim_{\cat C}(H_{\cat A})=\QQ\oplus \Prim(H_{\cat A})$. 
This proves the first assertion of Theorem~\ref{th:profinitary primitives}. It remains to prove
the second assertion (and thus complete the proof of Theorem~\ref{th:prim generation}), namely, that $\Prim(H_{\cat A})$ is 
a minimal generating space of~$H_{\cat A}$. We need the following result.
\begin{lemma}\label{lem:key-lem-quasi-Nichols}
Suppose that $A$ is both a unital algebra and coalgebra with 
$\Delta(1)=1\tensor 1$. Assume that $A$
admits a compatible pairing $\lr{\cdot,\cdot}:A\tensor A\to\FF$, in the sense 
of Definition~\ref{def:compatible pairing},
such that $\lr{a,1}=\varepsilon(a)$ for all $a\in A$. Let~$V=\Prim(A)$.
Then $1\notin V$ and $\lr{\sum_{k\ge 2} V^{k},\FF\oplus V}=0$.
\end{lemma}
\begin{proof}
Since $v\in V$ is primitive, $\varepsilon(v)=0$. Furthermore, we show that 
$\varepsilon:A\to \FF$ is a
homomorphism of algebras. Indeed, given $a,a'\in A$, we have
$$
\varepsilon(aa')=\lr{aa',1}=\lr{a\tensor a',\Delta(1)}=\lr{a\tensor a',1\tensor 
1}=\lr{a',1}\lr{a,1}=\varepsilon(a)\varepsilon(a').
$$
This immediately implies that $\varepsilon(V^\ell)=0$ and $\lr{V^\ell,V^0}=0$, 
$\ell>0$.
Finally, let~$v\in V$ and~$x,y\in\ker\varepsilon$. Then
\begin{equation}\label{eq:orth}
\langle xy, v\rangle=\langle x\tensor y,\Delta(v)\rangle=\langle x\tensor 
y,v\otimes 1+1\otimes v\rangle=
\langle y,v\rangle \varepsilon(x)+\varepsilon(y)  \langle x,v\rangle=0.
\end{equation}
Let~$\ell>1$. Since~$V^\ell\subset V\cdot V^{\ell-1}$ and $V^k\subset\ker\varepsilon$
for all~$k>0$, it follows that~$\lr{V^\ell,V}=0$, $\ell>1$.
\end{proof}
Let $V_{\cat A}=\Prim(H_{\cat A})$ and $(H_{\cat A})_{>1}=\sum_{r\ge 2} V_{\cat 
A}^r$.
By Lemma~\ref{lem:key-lem-quasi-Nichols}, $\lr{ (H_{\cat A})_{>1},\QQ\oplus 
V_{\cat A}}=0$. 
Since the pairing $\lr{\cdot,\cdot}$ on~$H_{\cat A}$ is symmetric positive 
definite, $(H_{\cat A})_{>1}\cap (\QQ\oplus V_{\cat A})=\{0\}$
hence 
the sum $(\mathbb Q\oplus V_{\cat A})+(H_{\cat A})_{>1}$
is direct. This proves the second assertion of Theorem~\ref{th:profinitary 
primitives} and completes the proof of Theorem~\ref{th:prim generation}.
\end{proof}

\subsection{Proof of Corollary~\ref{cor:prim-gen} and estimates for primitive elements}\label{subs:estimate}
\begin{proof}[Proof of Corollary~\ref{cor:prim-gen}]
Let~$R\subset H_{\cat A}^+:=\ker\varepsilon$ be a generating space for~$H_{\cat A}$. Then~$(H_{\cat A}^+)^\ell=\sum_{k\ge \ell} R^k$, $\ell\ge 1$.
Taking~$R=\QQ\Iso\cat A$ (Theorem~\ref{th:PBW-property}) and $R=\Prim(H_{\cat A})$ (Theorem~\ref{th:prim generation}) 
we conclude that $P=(H_{\cat A}^+)^2=\sum_{k\ge 2}\Prim(H_{\cat A})^k=\sum_{k\ge 2} (\QQ\Iso\cat A)^{k}$.
On the other hand, $H_{\cat A}^+=\Prim(H_{\cat A})+P$ and~$P\cap \Prim(H_{\cat A})=\{0\}$ by Lemma~\ref{lem:key-lem-quasi-Nichols}.
Therefore, 
$H_{\cat A}^+=\Prim(H_{\cat A})\oplus P$. 
The graded version is immediate.
\end{proof}

\begin{proof}[Proof of Proposition~\ref{prop:PBW-prec} and Lemma~\ref{lem:prim}] 
We need the following 
obvious fact from linear algebra.
\begin{lemma}
Let~$U$ be a finite dimensional $\FF$-vector space and $U_1,U_1',U_2\subset U$ be its subspaces such that 
$U=U_1+U_2=U_1'+U_2$. If~$U_1\cap U_2=\{0\}$ then $\dim_\FF U_1=\dim_\FF U_1'-\dim_\FF(U_1'\cap U_2)$. 
\end{lemma}
Taking into account Corollary~\ref{cor:prim-gen},
we apply this Lemma with~$U=(H_{\cat A})_\gamma$, $U_1'=\QQ\Ind\cat A_\gamma$, $U_2=P_\gamma$ and~$U_1=\Prim(H_{\cat A})_\gamma$
and obtain
$$
\dim_\QQ \Prim(H_{\cat A})_\gamma=\#\Ind\cat A_\gamma-\dim_\QQ (P_\gamma\cap \QQ\Ind\cat A_\gamma),
$$
which yields Proposition~\ref{prop:PBW-prec}. To prove Lemma~\ref{lem:prim}, note that $\QQ(\Iso\cat A\setminus \Ind\cat A)=
(\QQ\Ind\cat A)^\perp$. Thus, $\Prim(H_{\cat A})_\gamma\cap \QQ(\Iso\cat A\setminus \Ind\cat A)\
\subset (\QQ\Ind\cat A_\gamma)^\perp\cap P_\gamma^\perp=(\QQ\Ind\cat A_\gamma+P_\gamma)^\perp=(H_{\cat A})_\gamma^\perp=0$,
by Lemma~\ref{lem:lin-alg}\eqref{lem:lin-alg.b} and Corollary~\ref{cor:prim-gen}.
\end{proof}

\section{Proof of Theorem~\ref{th:hereditary nichols}}

\subsection{Diagonally braided categories}
We say that a bialgebra $H_0$ is co-quasi-triangular if it has a skew Hopf 
self-pairing ${\mathcal R}:H_0\otimes H_0\to \QQ$. 
Let ${\cat C}$ be the category of left $H_0$-comodules. This category is 
braided via the commutativity constraint $\Psi_{U,V}:U\otimes V\to V\otimes U$ 
for all objects $U,V$ of ${\cat C}$ defined by:
$$\Psi_{U,V}(u\otimes v)={\mathcal R}(u^{(-1)},v^{(-1)}) \cdot v^{(0)}\otimes 
u^{(0)}$$
for all $u\in U$, $v\in V$, where we used the Sweedler-like notation for the 
coactions: $\delta_U(u)= u^{(-1)}\otimes u^{(0)}$, $\delta_V(v)= v^{(-1)}\otimes 
v^{(0)}$. We will write~$\cat C_{\mathcal R}$ to emphasize that~$\cat C$ is a braided category.

\begin{remark} The category ${\cat C}_\chi$ introduced in 
Lemma~\ref{lem:bichar-category} is equivalent to the category of 
$H_0$-comodules, 
where $H_0=\QQ \Gamma$ is the  monoidal algebra of $\Gamma$ and ${\mathcal 
R}|_{\Gamma\times \Gamma}=\chi$.
\end{remark}

Our present aim is to prove the following result.
\begin{theorem}\label{prop:nichols-sub}
Let $B$ be a bialgebra in~$\cat C_{\mathcal R}$. 
\begin{enumerate}[{\rm(a)}]
 \item\label{prop:nichols-sub.a} The space $V=\Prim(B)$ is a subobject of~$B$ 
in~$\cat C_{\mathcal R}$;
\item\label{prop:nichols-sub.b} Suppose that~$B$ admits a compatible pairing, in the 
sense of Definition~\ref{def:compatible pairing},
such that $\lr{b,1}=\varepsilon(b)$ and $\lr{b,b}\not=0$ for all $b\in 
B\setminus\{0\}$. Then
the canonical inclusion $V\hookrightarrow B$ extends to an injective 
homomorphism 
\begin{equation}\label{eq:nichols-embed}
\mathcal B(V)\to B
\end{equation} 
of bialgebras in~$\cat C_{\mathcal R}$. In particular, if~$B$ is generated by~$V$,
then~\eqref{eq:nichols-embed} is an isomorphism.
\end{enumerate}
\end{theorem}
\begin{proof}
Part~\eqref{prop:nichols-sub.a} is a special case of the following
simple fact.
\begin{lemma}
If $C$ is a coalgebra in $\cat C_{\mathcal R}$ with unity, then 
 $V:=\Prim(C)$ is a subobject of $C$ in $\cat C_{\mathcal R}$.
 \end{lemma}

 \begin{proof} Denote by $\delta_C:C\to H_0\tensor C$ the left co-action of~$H_0$ on~$C$.
 All we have to show is that $\delta_C(V)\subset H_0\otimes 
 V$. Fix a basis $\{b_i\}$ of $H_0$ and let $v\in \Prim(C)$. Write 
 $$\delta_C(v)=\sum_i b_i\otimes v_i,\qquad v_i\in C.$$
 Since $\Delta:C\to C\otimes C$ is a morphism of left $H_0$-comodules,
 $$(1\otimes \Delta)\circ \delta_C(v)=\delta_C(v\otimes 1)+\delta_C(1\otimes 
 v).$$
 Taking into account that $\delta_C(1)=1\otimes 1$, we obtain
 $$\sum_i b_i\otimes \Delta(v_i)=\sum_i b_i\otimes v_i\otimes 1+\sum_i 
 b_i\otimes 1\otimes v_i\ ,$$
 which implies that
 $$\Delta(v_i)= v_i\otimes 1+ 1\otimes v_i,$$
that is, $v_i\in V$
 for all $i$. 
 \end{proof}

Now we prove~\eqref{prop:nichols-sub.b}. Denote by $B'$ the subalgebra of $B$ 
generated by $V=\Prim(B)$. 
It is sufficient to show that $B'=\mathcal B(V)$.  
We need the following result.
\begin{proposition}\label{prop:graded-alg}
$B'=\bigoplus_{k\ge 0} V^k$, hence $B'$ is a graded algebra. 
\end{proposition}
\begin{proof}
We prove that $\lr{V^\ell,V^k}=0$ for all $0\le k<\ell$ by  induction on the 
pairs $(k,\ell)$, $k<\ell$ ordered lexicographically. The induction base for 
$k=0,1$ is established in Lemma~\ref{lem:key-lem-quasi-Nichols}. Now,
fix~$\ell>2$ and suppose that $\langle V^s,V^r\rangle=0$ for all $r<s<\ell$. 
Let~$1<k<\ell$. Since $\Delta$ is a homomorphism of algebras,
$$
\Delta(V^k)\subset (V\tensor 1+1\tensor V)^k\subset \sum_{i=0}^k V^{k-i}\tensor 
V^i,
$$
hence
\begin{align*}
\langle V^\ell,V^k\rangle&\subset \langle V\tensor 
V^{\ell-1},\Delta(V^k)\rangle \subset \sum_{i=0}^k\langle V\otimes 
V^{\ell-1},V^{k-i}\otimes V^{i}\rangle\\
&=\sum_{i=0}^k\langle V,V^i\rangle\langle 
V^{\ell-1},V^{k-i}\rangle=\langle V,V\rangle \langle 
V^{\ell-1},V^{k-1}\rangle=\{0\}
\end{align*}
by the inductive hypothesis.
It remains to show that the sum $\sum_{k\ge 0}V^k$ is direct, which is an 
immediate consequence of the following obvious fact.
\begin{lemma}
Let~$U_i$, $i\in \mathbb Z_{\ge 0}$, be subspaces of an $\FF$-vector space~$U$ 
with a bilinear form 
$\lr{\cdot,\cdot}:U\tensor U\to \FF$ such that $\lr{U_j,U_i}=0$ if~$j>i$ and 
$\lr{u,u}\not=0$ for all~$u\in U\setminus\{0\}$.
Then the sum $\sum_{i} U_i$ is direct.\qed
\end{lemma}
\noindent
This completes the proof of Proposition~\ref{prop:graded-alg}.
\end{proof}
Since~$B'_0=\QQ$ 
and~$B'_1=V=\Prim(B')=\Prim(B)$, $B'$ is the Nichols algebra of~$V$ 
by Definition~\ref{defn:nichols}. Theorem~\ref{prop:nichols-sub} is proved.
\end{proof}

\subsection{Proof of Theorem~\ref{th:hereditary nichols}}\label{pf:hereditary 
nichols}
We need the following reformulation of celebrated Green's theorem 
(\cite{Green}) for Hall algebras (see also~\cite{Walker}).
\begin{proposition}  
\label{pr:Hopf structures}
Let ${\cat A}$ be a finitary and cofinitary hereditary abelian category. Then
the Hall algebra $H_{\cat A}$ is a bialgebra in ${\cat C}_{\chi_{\cat A}}$ with 
the coproduct $\Delta$ given by \eqref{eq:coproductDelta} and the counit 
$\varepsilon$ given by \eqref{eq:counit}.
\end{proposition}
\begin{proof}
For every $[C],[C']\in\Iso\cat A$ we have
\begin{align*}
\Delta([C])\Delta([C'])&=(\sum_{[A],[B]}  F^{A,B}_C\cdot  [A]\otimes [B]) 
(\sum_{[A'],[B']}   F^{A',B'}_{C'}\cdot  [A']\otimes [B'])\\
&=\sum_{[A],[B],[A'],[B']}  F^{A,B}_C F^{A',B'}_{C'} \cdot  \frac{|\Ext^1_{\cat 
A}(B,A')|} {|\Hom_{\cat A}(B,A')|} [A] [A']\otimes [B]  [B']\\
&=\sum_{[A],[B],[A'],[B'],[A''],[B'']}  F^{A,B}_C F^{A',B'}_{C'} 
F_{A,A'}^{A''}F_{B,B'}^{B''}\frac{|\Ext^1_{\cat A}(B,A')|} {|\Hom_{\cat 
A}(B,A')|}\cdot  [A'']\otimes [B'']
\end{align*}
On the other hand, 
$$\Delta([C][C'])=\sum_{[C'']}F_{C,C'}^{C''} \Delta([C''])=\sum_{[C''], [A''], 
[B'']}  F_{C,C'}^{C''} F^{A'',B''}_{C''}\cdot  [A'']\otimes [B''].$$
We need the following
\begin{lemma}[\cite{Green}*{Theorem~2}, see also~\cite{Schiffmann}] If $\cat A$ 
is finitary and cofinitary hereditary abelian category, then for any objects $A'',B'',C,C'$ of 
${\cat A}$ one has
\begin{equation}\label{E:Greenproof3}
\sum_{[A],[A'],[B],[B']}\frac{|\Ext^1_{\cat A}(B,A')|} {|\Hom_{\cat 
A}(B,A')|}\cdot F^{B''}_{B,B'}F^{A''}_{A,A'} 
F_{C}^{A,B}F_{C'}^{A',B'}=\sum_{[C'']}    F_{C,C'}^{C''} F_{C''}^{A'',B''}
\end{equation}
\end{lemma}
This immediately implies that~$\Delta([C])\Delta([C'])=\Delta([C][C'])$.
\end{proof}
Theorem~\ref{th:hereditary nichols} now follows from Proposition~\ref{pr:Hopf 
structures} and Theorem~\ref{prop:nichols-sub}.\qed

\end{document}